\documentclass[a4paper,12pt]{amsart}
\usepackage{xcolor}
\usepackage{amsfonts,amsmath,amsthm,bbold}
\usepackage{amssymb}
\usepackage{enumitem}
\usepackage{mathtools,upgreek}
\usepackage{esint}
\usepackage{nicematrix}
\usepackage[bookmarks,bookmarksopen=false,
                  colorlinks=true,%
                  linkcolor=blue,%
                  citecolor=blue,%
                  urlcolor=blue,%
                ]{hyperref}
\usepackage{geometry}

\newtheorem{lemma}{Lemma}[section]
\newtheorem{theorem}[lemma]{Theorem}
\newtheorem{corollary}[lemma]{Corollary}
\newtheorem{proposition}[lemma]{Proposition}
\newtheorem{definition}[lemma]{Definition}

\theoremstyle{remark}
\newtheorem{rem}[lemma]{Remark}

\newcommand{\nn}{\mathbb{N}}
\newcommand{\ee}{\mathbb{E}}
\newcommand{\rr}{\mathbb{R}}
\newcommand{\zz}{\mathbb{Z}}
\newcommand{\cc}{\mathbb{C}}
\newcommand{\one}{\mathbb{1}}
\newcommand{\zc}{\mathcal{Z}}
\newcommand{\qc}{\mathcal{Q}}
\newcommand{\pc}{\mathcal{P}}
\DeclareMathOperator{\pois}{Pois}
\DeclareMathOperator{\perm}{perm}
\newcommand{\logg}{{\log_*}}

\author[J. Buckley]{Jeremiah Buckley}
\address{Department of Mathematics\\ King's College London\\ United Kingdom}
\email{jeremiah.buckley@kcl.ac.uk}

\author[F. Marceca]{Felipe Marceca}
\address{Faculty of Mathematics \\
	University of Vienna \\
	Oskar-Morgenstern-Platz 1 \\
	1090 Vienna, Austria}
\email{felipe.marceca@univie.ac.at}

\author[J. Singer]{Joaqu\'in Singer}
\address{Departamento de Matem\'{a}tica, Facultad de Ciencias Exactas y Naturales, Universidad de Buenos Aires, (1428) Buenos Aires,
Argentina and IMAS-CONICET}
\email{jsinger@dm.uba.ar}

\title[Sampling properties of the zeroes of the GEF]{Sampling properties of the zeroes of the Gaussian entire function}

\thanks{This project was initiated during a research visit supported by LMS Scheme 5: 52309. The first two authors gratefully acknowledge support from EPSRC: NIA EP/V002449/1. F.M. was also supported by the Austrian Science Fund (FWF) 10.55776/PAT1384824 and 10.55776/Y1199. 
For open access purposes, the authors have applied a CC BY public copyright license to any author-accepted manuscript
version arising from this submission.}

\begin{document}

\begin{abstract}
We study sampling properties of the zero set of the Gaussian entire function on Fock spaces. Firstly, we relax Seip and Wallst\'en's density and separation conditions for sampling sets on Fock spaces to obtain weighted inequalities for sets that are not necessarily sampling. On the probabilistic front, we estimate the number of zeroes of the Gaussian entire functions that are close to each other. We use these to prove random sampling inequalities for polynomials of degree at most $d$ using ${d}+o(d)$ points, and show that, with high probability, the sampling constants grow slower than $d^\varepsilon$ for any $\varepsilon>0$.
In particular, we recover a result from Lyons and  Zhai in the case of the Gaussian entire function, where it is shown that the zeroes are (almost surely) a uniqueness set for the Fock space.
\end{abstract}

\maketitle

\section{Introduction}

\subsection{Statement of results}
Sampling theory deals with the (stable) reconstruction of a function $f$ from given samples $f(x_j)$ on some set $X=\{x_j\}_j\subseteq \rr^d$. We are interested in random sampling, that is, whenever $X$ is a random configuration of points. We work in $\rr^2$, which we identify with $\cc$. 

The Fock space, $\mathcal{F}_\alpha^p$, is defined for $\alpha>0$ and $1\le p<\infty$ to be the space of entire functions satisfying 
\begin{align}
    \label{eqf}
    \|f\|_{p,\alpha}^p=\frac{p \alpha}{2\pi}\int_{\cc} |f(z)|^p e^{-p \alpha|z|^2/2} dA(z)<\infty,
\end{align}
where $dA(z)$ is the Lebesgue measure on $\cc$. In what follows we set $\alpha=1$ and simply write $\mathcal{F}^p=\mathcal{F}_\alpha^p$ and $\|f\|_{p,\alpha} = \|f\|_p$. Naturally, all our results can be rescaled to cover the case $\alpha\neq 1$. A discrete set $\zc\subseteq\cc$ is a sampling set for $\mathcal{F}^p$ if there exist $A,B>0$ such that for every $f\in \mathcal{F}^p$,
\begin{align}
    \label{eqsa}
    A\|f\|_{p}^p\le \sum_{z\in\zc} |f(z)|^p e^{-p|z|^2/2} \le B\|f\|_{p}^p.
\end{align}

In \cite{Se,SeWa}, Seip and Wallst\'en characterised sampling sets $\zc$ in terms of density and separation. To be more precise, we say a set $\zc$ is separated if the distances between points in $\zc$ is uniformly bounded from below, and define its lower Beurling-Landau density as
\[D^-(\zc)=\liminf_{r\to\infty}\inf_{z\in\cc}\frac{\#(\zc\cap B_r(z))}{\pi r^2}.\]
As shown in \cite{Se,SeWa}, a set is sampling for $\mathcal{F}^p$ if and only if it is a finite union of separated sets and contains a separated subset $\zc'$  with density  $D^-(\zc')>\tfrac{1}{\pi}$. This value is referred to as the critical density. Additionally, $\zc'$ can be chosen to be a (separated) perturbation of a lattice $\Lambda$  with density  $D^-(\Lambda)>\tfrac{1}{\pi}$, where the perturbations are uniformly bounded (see for example \cite[Lemma 4.31]{Zhu}). So, sampling sets are a finite union of separated sets, one of which is a uniformly perturbed lattice above the critical density.

The quality of a sampling result can be measured in terms of how close to the critical density one can get (i.e., using a minimal number of points for sampling), and how close the ratio $B/A$ is to unity (here,  $A$ and $B$ are the ``sampling constants'' from \eqref{eqsa}); improving these quantities tends to improve the numerical efficiency of an algorithm to reconstruct a function from its samples. However, these are two competing factors since increasing the density of points (a.k.a. oversampling) tends to improve the ratio $B/A$. 

There are already a number of results on random sampling, in various contexts (see Section \ref{secba} for more details).  These results often use processes with embedded independence, such as Poisson or binomial processes. This is mainly for technical reasons and simplicity of implementation. In this work we shall instead be interested in processes exhibiting a form of local repulsion between points, which are therefore more rigid and show less clumping than the above processes. We are mainly motivated by certain sampling problems where the underlying point process is a predetermined repulsive point process, and we view our work as a first step towards understanding such problems. Furthermore, given that separation of points plays an important role in sampling, it is natural to consider repulsive processes.

Returning to the Fock space, for many natural choices of (locally) repulsive point processes, the density and separation requirements for sampling are not satisfied. This is because, if the correlation between points of the process decays with distance, then this gives rise to some form of weak independence over long distances. Heuristically, this independence guarantees that global conditions fail almost surely in some region. In particular, there will be arbitrarily large regions with no points or with too many points that are close together, violating the density and separation conditions that characterize sampling sets. Since one cannot expect global sampling inequalities, one has to work locally. This can be achieved by restricting the analysis to a finite dimensional subspace of functions or, more generally, a subset of localised functions that can be well approximated by a finite dimensional subspace. This is known as relevant sampling, see, for example, \cite{BaGr3}.

We focus on the zero set $\zc^{(L)}$ of the Gaussian entire function (GEF) given by 
\begin{align}
\label{eqgef}
    F_L (z)=F(z)=\sum_{n=0}^\infty \zeta_n \frac{(\sqrt{L} z)^n}{\sqrt{n!}},
\end{align}
where $\zeta_n$ are iid standard complex normal variables and $L>0$ is a scaling parameter. This is a particular case of a Gaussian analytic function (GAF), see \cite{HKPV} and Section~\ref{subgaf}. The distribution of the zeroes is invariant under translations and rotations, and $F_L$ is (essentially) the only GAF with this property. Additionally, almost surely all the zeroes are simple, they exhibit local repulsion, and are described in \cite{SoTs2} as a perturbed lattice. In \cite{LyZh}, Lyons and Zhai provide sufficient conditions for the zeroes of a GAF to be a uniqueness set in $\mathcal{F}^p_\alpha$ (which is a necessary condition to be  sampling) and related Bergman spaces (see also \cite{BuQiSh} for an analogous result for determinantal point processes). While this suggests that $\zc^{(L)}$ might be a sampling set for the Fock space, as discussed above, the global characterisation of sampling cannot hold. This is because the decay of the (normalised) covariance kernel
\begin{equation*}
    \frac{|\ee[F_L(z)\overline{F_L(w)}]|}{\sqrt{\ee[|F_L(z)|^2]\ee[|F_L(w)|^2]}}=e^{-L|z-w|^2/2}
\end{equation*}
means that the zeroes can be thought of (at least heuristically) as being independent on large scales. Nevertheless, we are able to work locally and provide quantitative estimates that, in particular, recover the uniqueness result from \cite{LyZh} for the GEF (see Corollary \ref{corolz}). We say $F_L$ has intensity $L/\pi$ which is the expected number of zeroes per unit area. For our local sampling inequalities, a role similar to that of the lower density in the global deterministic case will be played by the intensity\footnote{For instance, one can use ergodicity to show that, almost surely, $\frac{\#(\zc^{(L)}\cap B_r(z))}{\pi r^2}\to L/\pi$ as $r\to\infty$. Notice, however, that this statement is local since we are not taking the infumum over $z\in\cc$ before taking the limit.}, so we consider $L>1$. The choice of the GEF is due to some technical advantages, including the fast decay of correlations and, crucially, the perturbed lattice description. However, this work is meant as a first step towards the investigation of more general GAFs and other natural random processes exhibiting local repulsion such as determinantal point processes and Coulomb gases. More generally, we are interested in random sampling both as an alternative approach to sampling problems where few deterministic results are available, as well as a tool to study specific examples where the random process used for sampling is given or even is itself the object of study.

Denote $\logg(x)=\max\{1,\log(x)\}$. Our main result reads as follows.
\begin{theorem}
\label{teofull}
    Let $1 < L <\infty$, $1\le p <\infty$ and consider $\zc^{(L)}$, the zero set of $F_L$. There exist $c,C,C'>0$ depending on $L$ and $p$ satisfying the following: If $R\ge C$ then, outside an event of probability at most $e^{-c\logg R\logg\logg R}$,  
 \begin{align}
 \label{eqfull}
      \|f\|_{p}^p \le C'  \sum_{z\in \zc^{(L)}}\omega(|z|+R)|f(z)|^pe^{- p|z|^2/2}
 \end{align}
 for every $f\in \mathcal{F}^p$, where $\omega(r)=e^{C\sqrt{\logg r}\logg^6 \logg r}$. Moreover, if $L<2$ then one can choose the constants $c$ and $C$ depending only on $p$ and take $\logg\logg C'=C''/(L-1)$ with $C''\ge 1$ depending only on $p$. \end{theorem}
 \begin{rem}\label{rerem}Let us make a few comments:
    \begin{enumerate}[label=(\roman*)]
        \item Theorem \ref{teofull} should be regarded as a local sampling inequality in the following sense. If one restricts to a family of functions that are sufficiently concentrated on a bounded set, then the sum over $\zc^{(L)}$ in \eqref{eqfull} can be effectively truncated and the weight $\omega$ replaced by a constant. We carry out this procedure for polynomials of bounded degree in Theorem \ref{teosam}.
        \item Although \eqref{eqfull} only provides a weak version of the left-hand side of \eqref{eqsa}, a converse inequality follows easily from the subharmonicity of $|f|^p$ and probabilistic estimates for the number of points in balls from \cite{NaSoVo}. Again, we carry this out directly for polynomials in Theorem \ref{teosam} (see also Proposition \ref{propbe}).
        \item The weight $\omega(r)$ grows quite slowly (it is $O(r^\varepsilon)$ for every $\varepsilon>0$). Moreover $\omega(r)$ is optimal except for $\logg \logg r$ terms in the exponent (see Remark \ref{remopt}).
        \item  The constant $C'$ is superfluous if an explicit dependence on $L$ is not required. However, tracking the behaviour of the constants as $L\to 1$ is important to quantify oversampling, since the closer $L$ is to 1, the fewer the number of points we are using for sampling (see Theorem \ref{teosam}). On the other hand, tracking the dependence as $L\to\infty$ is not really relevant to our argument and we impose the arbitrary upper bound $L<2$.
        \item When rescaling for $\mathcal{F}_\alpha^p$ with $\alpha \neq 1$, the constant $C'$ in Theorem \ref{teofull} depends on $L - \alpha$ instead, while the remaining constants depend on $p$ and $\alpha$.
        \item  The proof of the theorem can be divided into a deterministic statement, and the verification that the zeroes of the GEF satisfy its hypotheses with high probability. Our deterministic result, Proposition \ref{propseip}, relaxes Seip and Wallst\'en's sufficiency result, \cite[Theorem 1.1]{SeWa}. It applies to  sets that need not be sampling, and produces weighted inequalities whose quality depends on local separation and density properties of the set.
        From a probabilistic perspective, we rely on the perturbed lattice description of the zeroes of the GEF from \cite{SoTs2}. We also establish estimates on the separation between zeroes of the GEF that may be of independent interest (see Sections \ref{sectec} and \ref{secse}).
    \end{enumerate}
 \end{rem}

As an immediate consequence of Theorem \ref{teofull}, we recover Lyons and Zhai's uniqueness result from \cite[Theorem 1.1]{LyZh} in the case of the GEF and for $L> 1$ (cf. \cite{ChLyPa,BuQiSh}).

\begin{corollary}
    \label{corolz}
    Let $ 1< L <\infty$, $1\le p <\infty$ and consider $\zc^{(L)}$, the zero set of $F_L$. Almost surely, the only function $f\in \mathcal{F}^p$ such that $f(\zc^{(L)})=\{0\}$ is $f=0$.
\end{corollary}

It is worth noting that Lyons and Zhai's result also holds in the critical case $L= 1$. Additionally, it deals with arbitrary GAFs.

Next we turn our attention to sampling constants for polynomials. Specifically, for a family $\{\zc_d\}_{d\in \nn}$ of finite subsets of $\cc$, we study the asymptotic behaviour of the best  constants $A_d$ and $B_d$ such that for every polynomial $f$ of degree at most $d$,
\begin{align}
    \label{eqmz}
    A_d \|f\|_{p}^p\le \sum_{z\in\zc_d} |f(z)|^p e^{-p |z|^2/2} \le B_d \|f\|_{p}^p.
\end{align}
Here, by ``best constants'' we mean those whose ratio $B_d/A_d$ closest to $1$. Whenever $B_d/A_d$ is uniformly bounded, these inequalities are sometimes called Marcinkiewicz-Zygmund inequalities and have been studied in the context of Fock spaces in \cite{GrOC}. 

Given the Gaussian weight in the definition of the Fock norms, polynomials of unit norm and degree at most $d$ on the Fock space are mainly concentrated on a ball with centre $0$ and radius approximately $\sqrt{d}$. Consequently, we work with $\zc_d=\zc^{(L_d)}\cap B_{R_d}(0)$, for suitable choices of $L_d$ and $R_d$. Using the perturbed lattice description for the zeroes of the GEF from \cite{SoTs2} and the sampling results from \cite{SeWa} it is easy to obtain Marcinkiewicz-Zygmund inequalities, if we slightly oversample by using $O(d (\logg d)^{1/2+\varepsilon})$ points. Our result is the following.

\begin{proposition}\label{propos}
    Let $\varepsilon>0$, $d\in \nn$ and choose $L_d\ge (\logg d)^{1/2+\varepsilon}$ and $ R_d=\sqrt{2d}$. Define $\zc_d=\zc^{(L_d)}\cap B_{R_d}(0)$. Given $1\le p <\infty$ there exist positive constants $A=A(p),B=B(p),c=c(\varepsilon,p)$ and $C=C(\varepsilon,p)$ satisfying the following. Outside an event of probability at most $C e^{-c(\logg d)^{(1+\varepsilon/2)}}$, the estimate \eqref{eqmz} holds with $B_d/A_d$ uniformly bounded. In particular, for $\delta>0$ there exists $d_0=d_0(\delta,\varepsilon,p)\in \nn$ such that outside an event of probability $\delta$ these inequalities hold simultaneously for every $d\ge d_0$. 
\end{proposition}

\begin{rem} A few comments are in order:
\begin{enumerate}[label=(\roman*)]
    \item If we choose $L_d=O((\logg d)^{1/2+\varepsilon})$ then we are oversampling by less than a logarithmic factor since $\#\zc_d= O(d (\logg d)^{1/2+\varepsilon})$ outside a set of probability $O(e^{-c d^2 (\logg d)^{1+2\varepsilon} })$ by \cite[Theorem 2]{SoTs3} (cf. \cite[Proposition 2]{BaGr3} and see Remark \ref{remrel}).
    \item The choice of $R_d$ is somewhat arbitrary and only serves to simplify the proof. Since we are already oversampling (due to the condition $L_d\ge (\logg d)^{1/2+\varepsilon}$), there is no appreciable gain in the sample size when trying to set $ R_d^2$ as small as possible. Nevertheless, the proof works for $R_d^2\ge d + C \sqrt{d\logg d}$.
    \item It is possible to replace the sum over $\zc_d$ by the sum over all of $\zc^{(L_d)}$ strengthening the second inequality from \eqref{eqmz}. This is because polynomials of unit norm and degree at most $d$ are essentially concentrated in $B_{\sqrt{d}}(0)$. The argument to show this is analogous to the proof of Proposition \ref{propbe}.
    \item For $p=2$ one can choose the ratio $B_d/A_d$ arbitrarily close to 1, provided that $d_0$ is large enough (see Remark \ref{remone}).
\end{enumerate}
\end{rem}

The situation is much more delicate when one samples with $O(d)$ points. In this case, the ratio $B_d/A_d$ in \eqref{eqmz} is not uniformly bounded, but grows quite slowly. As a consequence of Theorem \ref{teofull}, we have the following result.

\begin{theorem}\label{teosam}
 Let $ 1\le p <\infty$ and $d\in \nn$. There exist $c,C>0$ depending on $p$ satisfying the following. Outside an event of probability at most $C e^{-c\logg d\logg\logg d}$ we have
 \begin{multline*}
      e^{-C\sqrt{\logg d}\logg^6 \logg d}\|f\|_{p}^p \le \sum_{z\in \zc_d}|f(z)|^pe^{-p|z|^2/2}
     \\ \le  \sum_{z\in \zc^{(L_d)}}|f(z)|^pe^{-p|z|^2/2} \le C \sqrt{\logg d}\|f\|_{p}^p,
 \end{multline*}
for every polynomial $f$ of degree $\deg f \le d$, where $\zc_d=\zc^{(L_d)}\cap B_{R_d}(0)$, $1 +\tfrac{C}{\logg\logg d}\le L_d<2$ and $R_d^2=d+C\sqrt{d \logg d}$. In particular, for $\delta>0$ there exists $d_0=d_0(\delta,p)\in \nn$ such that outside an event of probability $\delta$ these inequalities hold simultaneously for every $d\ge d_0$. 
\end{theorem}

Similar to before, we have a good control of the number of points we are using for sampling. Notice that if we take $L_d=L< 1+\varepsilon$ fixed, then by \cite[Theorem 1]{NaSoVo} we have $\#\zc_d\le (1+\varepsilon)d+o(d)$ outside a set of probability $O(e^{-c_\varepsilon d^{2-\varepsilon}})$. In other words, we are essentially oversampling by a factor of $1+\varepsilon$. Moreover, choosing $L_d= 1 +\tfrac{C}{\logg \logg d}$, \cite[Theorem 1]{NaSoVo} ensures that $\#\zc_d\le d+o(d)$ outside a set of probability $O(e^{-c_\varepsilon d^{2-\varepsilon}})$.

Let us also mention that the sampling constants are sharp up to $\log \log$ factors accompanying $\sqrt{\logg d}$. These quantities are related to hole and overcrowding probabilities for $\zc_d$ (see  Remark \ref{remopt}).

Finally, let us briefly discuss interpolation, which can be thought of as the ``dual concept'' to sampling. A discrete set $\{z_j\}_j=\zc\subseteq\cc$ is called an interpolation set for $\mathcal{F}^p$ if there is a constant $N>0$ such that for every sequence $(a_j)_j\in\ell_p$ there exists $f\in \mathcal{F}^p$ with $f(z_j)e^{-|z_j|^2}=a_j$ and $\|f\|_{p}\le N \|(a_j)_j\|_p$.
From \cite[Lemma 6.1]{Se}, it follows that there is a constant $C>0$ such that for any $z,z'\in\zc$,
\begin{align}
    \label{interpol}
    1\le C N |z-z'|.
\end{align}
In other words, the interpolation constant can be bounded from below in terms of the minimum separation between the points in $\zc$. As before, although a global interpolation inequality might be too restrictive for random point processes in general, one could ask if a local interpolation inequality for polynomials of degree at most $d$ holds with high probability and with interpolation constants that grow slowly with $d$. However, this is not the case in general. For example consider $\zc^{(L)}$, the zero set of $F_L$. If we restrict our analysis to $\zc^{(L)} \cap B_{R}(0)$ one expects to find many pairs of points at a distance at most $C R^{-1/2}$ from each other (see \cite[Theorem 2]{feng2024smallest}). Here one should think of $R \approx \sqrt{d}$ so that there are about $d$ points in $\zc^{(L)} \cap B_{R}(0)$. From \eqref{interpol}, we see that $N\gtrsim d^{1/4}$. In other words, if we only wish to interpolate at the points in $\zc^{(L)} \cap B_{R}(0)$, and even if we can choose any function $f\in \mathcal{F}^p$ to interpolate (rather than just polynomials of bounded degree), we get a ``bad'' constant that grows at least like $R^{1/2}$.

\subsection{Background and related work}\label{secba}
The origins of random sampling can be traced back to the study of systems of exponentials whose frequency profile is a perturbation of a deterministic set (see, e.g., \cite{SeUl,ChLy,ChLyPa}). Beyond serving as a probabilistic approach to classical questions in frame theory, random sampling has proven useful in several areas. Examples include: sampling of sparse polynomials and the theory of compressed sensing (see, e.g., \cite{CaRoTa,CaTa,Donoho}); theoretical frameworks to study the performance of certain numerical algorithms in signal processing (see, e.g., \cite{BaGr,KuRa}); and sampling in higher dimensions and relevant sampling (see, e.g., \cite{BaGr2,BaGr3,FuXi,LiSuXi}). 

Our approach is based on reducing the problem to a deterministic one by showing that the zeroes of the GEF satisfy certain deterministic conditions with high probability.
This is in contrast to other more nuanced probabilistic approaches from random sampling that can be applied even in contexts where the deterministic theory is less developed (e.g. multivariate sampling). The challenge we address lies elsewhere, and arises from the interactions between nearby zeroes. The main motivation behind this is to understand the sampling properties of point processes exhibiting repulsion and therefore less clumping than Poisson processes and other constructions with built-in independence.

The most natural examples of point processes that do not artificially restrict clumping include determinantal processes, Coulomb gases, and the zeroes of GAFs. 
Heuristically, less clumping and more rigidity should improve the sampling properties of the process. Furthermore, there are contexts where the repulsive process is the object of study or part of the problem's setting and cannot be chosen freely. We mention, for example, the study of sampling for Coulomb gases (see \cite{AmRo}), and wireless sensor networks (see \cite{ZaCo}). We therefore view this work as a first step towards this general goal.  We chose to study the zeroes of the GEF, since they offer some technical advantages such as their description as a perturbed lattice in \cite{SoTs2}, which is a key tool in our approach. This is not available for other point processes, and different approaches may be needed. Alternatively, one could try to prove a perturbed lattice result for these processes and apply our techniques; see \cite[Section 4.2]{ElSpYa} for another motivation for proving such a result in the Ginibre case.

Global sampling is achievable for random sets that retain some structure over large distances. See, for example, \cite{AnCaRo} (cf. \cite{RaSi}) for sampling in shift-invariant spaces using unions of random translates of lattices. However, as mentioned before, random sampling on unbounded sets is not feasible whenever the correlation between points decays with their distance. In relevant sampling, functions are assumed to be concentrated on a compact set and sampled on a finite number of random points near this set. In our case, we argue similarly by restricting to polynomials of bounded degree. These two approaches are quite related, given that the monomials are the eigenvectors of the concentration operator on a disc centred at the origin with eigenvalues decreasing to zero as the degree of the monomials increases (see Section \ref{sec2}). Using this and Proposition~\ref{propos}, one could proceed as in \cite{BaGr3} to produce a relevant sampling result. Doing the same in the context of Theorem \ref{teosam} would require a stricter concentration assumption for the family of functions where relevant sampling is carried out, namely a decay condition for $\int_{B_R(0)^c}|f|^p e^{-p\alpha|z|^2/2}$ as $R\to \infty$, rather than assuming this magnitude to be small for a fixed $R$. 

In broad terms, random sampling inequalities involve a discretization of an integral in terms of a random discrete set. From a probabilistic perspective, and specifically in the study of random point processes, this is related to the study of fluctuations of linear statistics. This has been extensively studied for zeroes of GAFs (see, e.g., \cite{SoTs,NaSo,HKPV}). 
An important difference between  these results and ours is that they deal with a fixed function.
In contrast, in random sampling one estimates the probability that for every function in a family the integral can be discretized uniformly. 
Nevertheless, it should be possible to combine a result for a fixed function with a metric entropy argument as in \cite[Theorem 4.7]{BaGr} to produce a random sampling result. For example, one could use the quantitative Offord-type estimate \cite[Theorem 7.1.1]{HKPV} (see also \cite{Of,So}) for fluctuations of a compactly supported $\mathcal{C}^2$ function as a starting point. We did not pursue this strategy, however, partly because we expect it would yield a weaker result. We believe that it would lead to oversampling by a large factor, due to the metric entropy argument. This is the case for \cite{BaGr}, which was later improved by the same authors to the suspected optimal oversampling factor in \cite{BaGr3} by avoiding the metric entropy argument.

Finally, as mentioned before, our results can be seen as a quantitative version of Lyons and Zhai's uniqueness result from \cite[Theorem 1.1]{LyZh} for the GEF (cf. \cite{ChLyPa,BuQiSh}).

\subsection{Technical overview}\label{sectec}
Our results are based on reducing the problem to a deterministic one by showing that the zeroes of the GEF satisfy certain deterministic conditions with high probability. In other words, the work can be divided into two parts: finding a suitable deterministic statement, and proving that the zeroes of the GEF satisfy its hypotheses with high probability. 

In the case of Proposition \ref{propos}, the reduction to the deterministic case is a straightforward procedure since all of the necessary ingredients are already available to us. By \cite{SoTs2}, the zeroes of the GEF can be described as a perturbed lattice. This perturbation ``decreases'' as we increase the intensity of the GEF, since changing the intensity is equivalent to rescaling the GEF (see \eqref{eqgef}). By oversampling with the GEF of high intensity, we can therefore ensure a small enough degree of perturbation with respect to a lattice within a bounded set. In essence we are only dealing with the zeroes from a macroscopic point of view. This is enough to construct associated sampling sets and use Seip and Wallst\'en's theorem on sufficiency for sampling in the Fock space, \cite[Theorem 1.1]{SeWa}. 

For Theorems \ref{teofull} and \ref{teosam}, the situation is much more subtle. In this case, the intensity $L/\pi$ of the GEF can be close to the critical density $1/\pi$ for sampling sets on $\mathcal{F}^p$. Unlike before, here we have to deal with local interactions between the zeroes of the GEF and understanding the separation between points becomes important. We have to work both on the deterministic and on the probabilistic front.

Firstly, we provide a quantitative version of Seip and Wallst\'en's sampling theorem, Proposition \ref{propseip}. The main advantage is that it allows for milder assumptions and the strength of the result varies according to the characteristics of the discrete set involved. It can be applied to non-sampling sets $\zc\subseteq \cc$ for which it produces inequalities of the form
\[\|f\|_{p}^p\le C \sum_{z\in\zc} \omega(z) |f(z)|^p e^{-\frac{p}{2}|z|^2}, \quad f\in \mathcal{F}^p,\]
where $\omega$ is a weight depending on $\zc$. To be more specific, $\omega$ depends on how well separated the points in $\zc$ are and how far $\zc$ is from being a lattice. Crucially, Proposition \ref{propseip} allows for $\zc$ to have a small proportion of badly separated points. The overall strategy to prove Proposition \ref{propseip} is similar to Seip and Wallst\'en's theorem, but the argument is much more involved from a technical point of view. We mention that the approach of \cite{BeOC} may offer another way of producing similar estimates.

Secondly, we use \cite{SoTs2} again to describe the zeroes of the GEF as a perturbed lattice. In order to apply Proposition \ref{propseip} we need to control the separation between the zeroes of the GEF. We refer to \cite{feng2024smallest} for a deep structural result describing distances between zeroes. For our purposes, however, we need quantitative estimates. We provide bounds for the proportion of badly separated zeroes in a disc of arbitrary radius. For large radii, we use the almost independence of the GEF over large distances from \cite{NaSo} (see also \cite{NaSoVo2,NaSoVo}). For small radii, we have to deal with dependence. In Lemma \ref{lemrho}, we establish a quantitative version of the upper bound in \cite[Theorem 1.1]{NaSo3} to estimate the $k$-point intensity function of a GAF. This result (which may be of independent interest) allows us to control the number of badly separated points at small distances in Proposition \ref{coroloca}.

\subsection{Organization}
Section \ref{sec2} includes some basic facts about Fock spaces, Gaussian analytic functions and some auxiliary formulas for determinants associated to covariance matrices. In Section \ref{sec3} we prove Proposition \ref{propos}. In Section \ref{secsa} we prove Proposition \ref{propseip}, a quantitative version of Seip and Wallst\'en's sampling theorem \cite[Theorem 1.1]{SeWa}. Section \ref{secse} deals with separation between zeroes of the GEF, including a quantitative version of the upper bound in \cite[Theorem 1.1]{NaSo3} to estimate the $k$-point intensity function of a GAF, Lemma \ref{lemrho}. In Section \ref{secGEF} we prove Theorems~\ref{teofull} and \ref{teosam}. Finally, in Appendix \ref{applem} we prove Lemma \ref{lemlow}, a technical step in the proof of Proposition~\ref{propseip}.

\section{Preliminaries}
\label{sec2}
\subsection{Notation} We use $c,C$ to denote constants that only depend on the parameter $p$ and that may change from line to line. As mentioned before, we write $\logg x=\max(1,\log x)$. We denote $\nn_0=\nn\cup \{0\}$.

\subsection{Fock spaces} 
Beyond the Fock spaces defined in \eqref{eqf} we sometimes make use of the Fock space $\mathcal{F}^\infty$ of entire functions $f$ with finite  $\infty$-norm given by
\[\|f\|_{\infty}=\sup_{z\in\cc} |f(z)|e^{-|z|^2/2}.\]

Let us introduce some basic facts regarding Fock spaces. 
For $1\le p\le\infty$ and $a\in \cc$, define the Bargmann-Fock shift as the translation operator $\mathcal{T}_a:\mathcal{F}^p\to \mathcal{F}^p$ given by
\begin{align}
    \label{eqshi}
\mathcal{T}_{a}f(z) = e^{\overline{a}z - |a|^2/2}f(z-a),\quad f\in \mathcal{F}^p.
\end{align}
The  operator $\mathcal{T}_a$ is an isometric isomorphism (see \cite[Section 2.6]{Zhu}).
 
The map $p\mapsto \|\cdot\|_{p}$ is decreasing in $p$ (see \cite[Corollary 2.8]{Zhu}). Moreover, given $f$ holomorphic and $p>0$, by the subharmonicity of $|f(z)|^p$  and translation invariance, there is a constant $C>0$ depending on $p$ such that 
\begin{align}
\label{eqloc}
    |f(z)|^p e^{-p |z|^2/2}\le C \int_{B_1(z)} |f(w)|^p e^{-p|w|^2/2} dw.
\end{align}

The normalized monomials 
\[e_n(z)=\frac{z^n}{\sqrt{n!}},\]
form an orthonormal basis of $\mathcal{F}^2$. They are the images of the Hermite functions under the Bargmann transform (see \cite[Proposition 2.1, Theorem 6.8]{Zhu} and \cite[Section 3]{Gro} for further details). They are the eigenfunctions of the concentration operators given by 
\[T_R=Q\chi_{B_R(0)}Q,\]
where $Q:L^2(\cc,e^{- |z|^2}dA(z))\to \mathcal{F}^2$ is the orthogonal projection onto $\mathcal{F}^2$ Moreover, their corresponding eigenvalues are given by
\begin{align}
    \label{eqei}
    \lambda_n(R)=\frac{1}{n!}\int_0^{ R^2}x^ne^{-x}dx=1-e^{- 
 R^2}\sum_{k=0}^n \frac{R^{2k}}{k!}.
\end{align}
In particular, notice that 
\[\lambda_n(R)=P(X_R>n),\]
where $X_R\sim \pois( R^2)$. We will use the following Chernoff bound (see for example \cite[Section 2.2]{BoLuMa}):
\begin{align}
    \label{eqche}
    (1-\lambda_d(R)) =P(X_R\le d) \le e^{-( R^2-d-d\log(R^2/d))}.
\end{align}
Note that $R^2-d-d\log(\tfrac{R^2}{d})$ is positive and increasing in $R$ if $R^2>d$.

The following lemma allows us to estimate the norms of polynomials outside a centred disc (cf. \cite[Lemma 2.2]{GrOC}).
\begin{lemma}[Tail estimates]
\label{lemtai}
    For $d\in\nn$ let $f$ be a polynomial of degree at most $d$ and consider $A=B_r(0)^c$ where $d< r^2$. Then, for every $1\le p<\infty$ there is a constant $C>0$ depending on $p$ such that
    \begin{align*}
         \int_A |f(z)|^p e^{-p |z|^2/2} dA(z) \le C   r^{2}e^{-( r^2-d-d\log( r^2/d))/2}  \|f\|_{p}^p.
    \end{align*}
\end{lemma}

\begin{proof}
    We start with the case $p=2$. For $f=\sum_{n=0}^d a_n e_n$, 
    \begin{align*}
        \frac{ 1}{\pi}\int_A |f(z)|^2 e^{- |z|^2} dA(z) &= \langle (I-T_r) f,f\rangle = \sum_{n=0}^d (1-\lambda_n(r)) |a_n|^2
        \\ &\le (1-\lambda_d(r)) \|f\|_{2}^2 
 \le e^{-( r^2-d-d\log(r^2/d))}\|f\|_{2}^2,\notag
    \end{align*}
where in the last inequality we used \eqref{eqche}.

    Now assume that $p\le 2$ and for $n\in\nn$ define $A_n=\{2^{n-1}r<|z|<2^nr\}.$ By H\"older's inequality and the monotonicity of the $p$-norms we have
    \begin{align*}
        \int_{A_n} |f(z)|^p e^{-p |z|^2/2} dA(z) &\le  |A_n|^{1-p/2} \Big(\int_{A_n} |f(z)|^2 e^{-|z|^2} dA(z)\Big)^{p/2} 
        \\ &\le C (2^n r)^{2-p}  e^{-( 4^{n-1}r^2-d-d\log( 4^{n-1}r^2/d))p/2}\|f\|_{2}^p
        \\ &\le C r^{2} e^{-(r^2-d-d\log(r^2/d))/2}\|f\|_{p}^p 4^{n+d(n-1)p/2}e^{-(4^{n-1}-1)dp/2 }.
    \end{align*}
    Summing in $n$ settles the case $p\le 2$.
    
    It remains to check the case $p\ge 2$. Similar to before,
    \begin{align}\label{eq3}
        \int_A |f(z)|^p e^{-p |z|^2/2} dA(z) &\le  \|f\|_{\infty}^{p-2} \int_A |f(z)|^2 e^{-|z|^2} dA(z) 
        \\& \le \|f\|_{p}^{p-2}  C e^{-(r^2-d-d\log( r^2/d))}\|f\|_{2}^2. \notag
    \end{align}
    Now let $r_0:=\sqrt{d}< r$ and notice that
    \begin{align}\label{eq4}
        \|f\|_{2}^2&= \sum_{n=0}^d |a_n|^2 \le \lambda_d(r_0)^{-1} \sum_{n=0}^d \lambda_n(r_0)|a_n|^2=\lambda_d(r_0)^{-1} \int_{B_{r_0}(0)} |f(z)|^2 e^{-|z|^2} dA(z)
        \\ &\le \lambda_d(r_0)^{-1} |B_{r_0}(0)|^{1-2/p} \|f\|_{p}^2\le C\lambda_d(r_0)^{-1} d^{1-2/p} \|f\|_{p}^2. \notag
    \end{align}
    Finally, observe that $\lambda_d(r_0)=P(X_{r_0}>d)$ where $X_{r_0}\sim \pois(d)$, so $\lambda_d(r_0)\to 1/2$ as $d\to\infty$ by the CLT for the Poisson process. In particular, $\lambda_d(r_0)$ is bounded from below. Combining this with \eqref{eq3} and \eqref{eq4} finishes the proof.
\end{proof}

\subsection{Gaussian analytic functions}\label{subgaf} 
A Gaussian analytic function (GAF) is a random variable $F$ taking
values in the space of holomorphic functions on a region $U\subseteq \cc$, such that $(F(z_1),\ldots,F(z_n))$ has a mean zero complex Gaussian
distribution for every $n \ge 1$ and every $z_1,\ldots, z_n \in U$.
For our purpuses we restrict our analysis to GAFs of the form
\begin{align*}
    F(z)=\sum_{n=0}^\infty a_n \zeta_n z^n,
\end{align*}
where $\zeta_n$ are iid standard complex normal variables, $a_n\ge 0$ for every $n\in \nn$ and $\limsup_n a_n^{1/n}<\infty$ (this is to ensure a positive radius of convergence, see \cite[Lemma~2.2.3]{HKPV}). Notice that rescaling a GAF $F(z)$ by a factor $r>0$ to get $F(rz)$, coincides with replacing its coefficients $a_n$ with $r^n a_n$. We use this frequently to rescale results for the GEF of intensity $1/\pi$ to the GEF of intensity $L/\pi$. In addition to the GEF we will use auxiliary GAFs in our analysis such as the hyperbolic GAF $\widetilde F$ given by
\[\widetilde F(z)=\sum_{n=0}^\infty  \zeta_n z^n,\]
whose zeroes enjoy a determinantal structure (see \cite{PeVi}).

For a random point process $\pc$ and $k\in\nn$, the $k$-point intensity function $\rho_{k}$ (alternatively $k$-factorial moment density or $k$-point correlation function, see \cite[Section 4.3]{MR3236788} and also \cite[Chapter 3]{HKPV}) is determined by the formula
\[\ee\Big[
\sideset{}{^{\neq}}\sum_{z_1,\ldots,z_k\in\pc}
f(z_1,\ldots,z_k)\Big]=\int_{\cc^k}f(w_1,\ldots,w_k)\rho_{k}(w_1,\ldots,w_k) \, dA^k(w_1,\ldots, w_k),\]
where $f$ is any nonnegative measurable function and the summation runs over all ordered $k$-tuples of distinct points in $\pc$. In the particular case that $\pc=\zc_F$ is the zero set of a GAF $F$, the $k$-point function (which we denote $\rho_{k,F}$) can be computed (see \cite[Section 3.4]{HKPV}) as
\begin{align}\label{eqrho}
    \rho_{k,F}(z_1\ldots,z_k)=\frac{1}{\pi^{2k}\det (\Gamma_F)}\int_{\cc^k} |\eta_1\ldots\eta_k|^2 e^{-\tfrac{1}{2}\langle \Gamma_F^{-1}\eta', \eta'\rangle}dA^k(\eta_1\ldots \eta_k),
\end{align}
where $\eta'=(0,\ldots,0,\eta_1,\ldots,\eta_k)\in \cc^{2k}$ and $\Gamma_F$ denotes the covariance matrix of the random vector $(F(z_1),\ldots,F(z_n),F'(z_1),\ldots,F'(z_n))$ (provided that $\Gamma_F$ is indeed invertible). We write $\rho_F=\rho_{k,F}$ when there is no room for confusion.

\begin{definition}\label{def: s}
    For a discrete set $\zc\in\cc$ and $z\in\zc$ define the separation of $z$ as the distance to its closest neighbour and denote it by 
    \begin{align*}    
    s_\zc(z)&=d(z,\zc\smallsetminus\{z\}).
    \end{align*}
    If $\zc_F$ is the zero set of a function $F$ we also write $s_F=s_{\zc_F}$ as well as $s^{(L)}=s_{\zc^{(L)}}$ for the zero set of the GEF $F_L$.
\end{definition}

\begin{lemma}\label{lemloc}
    There exists an absolute constant $C>0$ such that for $\tau,\varsigma>0$,
    \[P\big(\,\exists z\in \zc^{(1)}\cap [0,\tau)^2: \ s^{(1)}(z)<\varsigma \big)\le C \tau^2\varsigma^4.\]
\end{lemma}
\begin{proof}
Note that
\begin{align*}
    P(\exists z\in \zc^{(1)}\cap [0,\tau)^2: \ s^{(1)}(z)<\varsigma )
    &\le\ee[\#\{z\in Z^{(1)}\cap [0,\tau)^2: \ s^{(1)}(z)<\varsigma\}]
    \\&\le \ee\Big[\sideset{}{^{\neq}}\sum_{z_i,z_j\in\zc^{(1)}}1_{[0,\tau)^2}(z_i)1_{B_\varsigma(z_i)}(z_j)\Big] 
    \\&= \int_{[0,\tau)^2}\int_{B_\varsigma(z)}\rho_{2,F_1}(z,w)\, dA(w)dA(z)
    \\&\le C \int_{[0,\tau)^2}\int_{B_\varsigma(z)}|z-w|^2\, dA(w)dA(z)\le C \tau^2 \varsigma^4,
\end{align*}
where we used \cite[Theorem 1.1]{NaSo3} to estimate $\rho_{2,F_1}(z,w)$ (see also Lemma \ref{lemrho}).
\end{proof}

\subsection{Determinants associated to Cauchy and Vandermonde matrices}
\label{secdet}

To study the behaviour of the matrix $\Gamma_F$ we will compare it to Vandermonde and Cauchy type matrices which we now describe.
For $z,w\in\cc^k$, we write the associated Vandermonde and Cauchy matrices as
\begin{align*}
    V(z)=(z_i^{j-1})_{i,j=1}^k, \quad \text{and,}\quad C(z,w)=\Big(\frac{1}{1-z_i \overline{w_j}}\Big)_{i,j=1}^k.
\end{align*}
We also define the $2k\times 2k$ block matrices (here $(z,u)=(z_1,\ldots,z_k,u_1,\ldots,u_k)$)
\begin{align*}
    M_V(z)&=\left. \tfrac{\partial^{k}}{\partial u_1 \ldots \partial u_k}V(z,u) \right|_{u=z}
    =
    \begin{pmatrix}
        1 & z_1 & z_1^2 & \ldots & z_1^{2k-1}
        \\ \vdots & \vdots & \vdots & & \vdots
        \\ 1 & z_k & z_k^2 & \ldots & z_k^{2k-1}
        \\ \hline 0 & 1 & 2z_1 & \ldots & (2k-1) z_1^{2k-2}
        \\ \vdots & \vdots & \vdots & & \vdots
        \\ 0 & 1 & 2z_k & \ldots & (2k-1) z_1^{2k-2}
    \end{pmatrix},
\intertext{and,}
M_C(z)& =\left. \tfrac{\partial^{2k}}{\partial u_1 \ldots \partial u_k \partial \overline{v_1} \ldots \partial \overline{v_k}}C((z,u),(z,v)) \right|_{u=v=z}
   \\ &=
    \begin{pNiceArray}{c|c}
  \Big(\frac{1}{1-z_i \overline{z_j}}\Big)_{i,j=1}^k & \Big(\frac{z_i}{(1-z_i \overline{z_j})^2}\Big)_{i,j=1}^k \\
  \hline
   \Big(\frac{\overline{z_j}}{(1-z_i \overline{z_j})^2}\Big)_{i,j=1}^k & 
   \Big(\frac{1+z_i\overline{z_j}}{(1-z_i \overline{z_j})^3}\Big)_{i,j=1}^k
\end{pNiceArray}.
\end{align*}

Notice that 
\begin{align}\label{eqmat}
    M_VM_V^*=\Gamma_F, \quad \text{and,}\quad M_C=\Gamma_{\widetilde{F}},
\end{align}
where $F(z)=\sum_{n=0}^{2k-1}\zeta_n z^n$ and $\widetilde{F}(z)=\sum_{n=0}^{\infty}\zeta_n z^n$.
The following lemma provides explicit formulas for the determinant of these matrices, which can be understood through the theory of confluent matrices. The Vandermonde case is well-known and, after some rearranging, the Cauchy case can be computed using \cite[Lemma 7]{Va} (see also \cite[Lemma]{GrJo}), so we omit the proof. Alternatively, in this simple setting these determinants can be computed directly in terms of the classical Vandermonde and Cauchy determinants by passing the derivatives outside the determinant using multilinearity. 

\begin{lemma}
    \label{lemcon}
    For $z_1,\ldots,z_k\in\cc$ we have
    \begin{align*}
        \det M_V(z) = (-1)^{\tfrac{k(k-1)}{2}}
        \prod_{1\le i<j\le k} (z_j-z_i)^4, \quad \text{and,} 
        \quad \det M_C(z) = \frac{\prod_{1\le i<j\le k} |z_j-z_i|^8}{\prod_{i,j=1}^k (1-z_i\overline{z_j})^4}.
    \end{align*}
\end{lemma}

\subsection{Perturbed lattice description of the zeroes of the GEF}
For future reference, we briefly explain the description of the zeroes of the GEF as a perturbed lattice from \cite{SoTs2}. For $L>0$, consider the square lattice $\Lambda^{(L)}=\{\lambda_{mn}^{(L)}\}_{m,n\in\zz}\subseteq\cc$ of density $L/\pi$ given by,
\[\lambda_{mn}^{(L)}=\sqrt{\pi/L}(m+in).\]
We avoid the superscript $(L)$ when there is no room for confusion.

Recall that $F_1$ denotes the GEF with $L=1$ and let $\zc^{(1)}$ be its zero set.
Then the distribution of $\zc^{(1)}$ is the same as $\{\lambda_{mn}^{(1)}+\xi_{mn}\}_{m,n\in\zz}$, where the random variables $\xi_{mn}$ satisfy the following properties:
\begin{enumerate}[label=(\roman*)]
    \item The distribution of $\{\xi_{mn}\}_{m,n\in\zz}\sim \{\xi_{m+m_0,n+n_0}\}_{m,n\in\zz}$ is invariant under shifts $(m_0,n_0)\in\zz^2$. In particular, the variables $\xi_{mn}$ are identically distributed, but not independent (they are expected to exhibit almost independence at long distances and negative correlation at short ones).
    \item \label{eqexpe} $\ee[e^{\varepsilon|\xi_{00}|^2}]<\infty$ for some $\varepsilon > 0$.
\end{enumerate}

As was commented in \cite{NaSo2}, the bound \ref{eqexpe} (see \cite[Main Theorem]{SoTs2}) can be improved to 
    \begin{align*}
        P(|\xi_{00}|>\tau)\le C e^{-c\tau^4/\logg\tau}.
    \end{align*}
Note that this can be obtained applying \cite[Lemma 7]{Kr} to refine \cite[Lemma 5.1]{SoTs2}. 

It follows from \eqref{eqgef} that $\sqrt{L}\zc^{(L)}$ and $\zc^{(1)}$ have the same distribution.
By a slight abuse of notation, we will frequently write $\zc^{(L)}=\{z_{mn}\}_{m,n\in\zz}$ with $z_{mn}=\lambda_{mn}^{(L)} + L^{-1/2}\xi_{mn}$ and call $\lambda_{mn}^{(L)}$ the associated lattice point of $z_{mn}$. 
 In particular,
    \begin{align}\label{eqdev}
        P(\sqrt{L}|z_{00}|>\tau)\le C e^{-c\tau^4/\logg\tau}.
    \end{align}

\section{Proof of Proposition \ref{propos}}\label{sec3}
We start by showing Marcinkiewicz-Zygmund inequalities for the zeroes of the GEF when we oversample by a $(\logg d)^{1/2+\varepsilon}$ factor. At this scale the perturbed lattice model provides sufficient control to directly implement the results from \cite{SeWa}.

\begin{proof}[Proof of Proposition \ref{propos}]
    It suffices to show the claim for $d$ sufficiently large, since otherwise one can take the failure probability to be 1 and the claim is vacuous. We write $R,L$ instead of $R_d,L_d$. We write $\zc^{(L)}=\{z_{mn}\}_{m,n\in\zz}$ and $\Lambda=\{\lambda_{mn}\}_{m,n\in\zz}$ given by,
 \[\lambda_{mn}=\sqrt{\pi/L}(m+in).\]
 By a union bound, translation invariance and \eqref{eqdev} we have that 
    \begin{multline}\label{eqdev2}
        P(\exists m,n\in\zz:  \   \lambda_{mn}\in B_{2R}(0), |z_{mn}-\lambda_{mn}|>(\logg d)^{-\varepsilon/4})
        \\ \le C L d P(\sqrt{L}|z_{00}|>(\logg d)^{1/4+\varepsilon/4} )\le C e^{-c(\logg d)^{1+\varepsilon/2}}.
    \end{multline}
    We partition the lattice $\Lambda=\bigcup_{j\in J} \Lambda_j$ into $\#J=\lceil \sqrt{L/2}\rceil^2$ disjoint sublattices $\Lambda_j$, that are translations of the lattice $\lceil \sqrt{L/2}\rceil \Lambda$. Notice that these sublattices have density
    \[\frac{2}{\pi}\big(1+O(L^{-1/2})\big)>\frac{1}{\pi},\]
    provided that $d$, and therefore $L$, is sufficiently large. Consider the sets $\zc_j,\zc$ given by
    \begin{align*}
        \zc_j&=\{z_{mn}\in\zc^{(L)}: \ \lambda_{mn}\in \Lambda_j\cap B_{2R}(0)\}\cup (\Lambda_j\cap B_{2R}(0)^c).
        \\ \zc&=\bigcup_{j\in J}\zc_j=\{z_{mn}\in\zc^{(L)}: \ \lambda_{mn}\in B_{2R}(0)\}\cup (\Lambda \cap B_{2R}(0)^c).
    \end{align*}
    
    Choosing $d$ sufficiently large we get that outside a set of probability given by \eqref{eqdev2}, the sets $\zc_j$ are $(\logg d)^{-\varepsilon/4}$-uniformly close to the lattice $\Lambda_j$ (that is, $|z_{mn}-\lambda_{mn}|\le (\logg d)^{-\varepsilon/4}$ for every $z_{mn}\in\zc_j$). Making $d$ larger if necessary and applying \cite[Theorem 1.1]{SeWa} we get
    \begin{align*}
    \|f\|_{p}^p\le C\sum_{z\in \zc_j}|f(z)|^pe^{-p|z|^2/2}, \quad f\in \mathcal{F}^p, j\in J,
\end{align*}
where the constant $C$ only depends on $p$ and may change from line to line. Summing in $j$, we have
\begin{align*}
     \|f\|_{p}^p\le C L^{-1}\sum_{z\in \zc}|f(z)|^pe^{- p|z|^2/2}, \quad f\in \mathcal{F}^p.
\end{align*}

 Notice that $\zc\cap B_R(0)\subseteq\zc^{(L)}\cap B_R(0)$.  So, using \eqref{eqloc} for polynomials $f$ with $\deg f \le d$,
\begin{align*}
    \|f\|_{p}^p &\le CL^{-1} \sum_{z\in \zc^{(L)} \cap B_R(0)}|f(z)|^pe^{- p|z|^2/2} +C \int_{B_{R-1}(0)^c} |f(z)|^pe^{-p|z|^2/2} dA(z)
\end{align*}
Since $d$, and therefore $R$, is sufficiently large we can assume that 
\[(R-1)^2\ge \tfrac{3}{4} R^2=\tfrac{3}{2}d.\] 
Recalling that $ x^2-d-d\log(\tfrac{ x^2}{d})$ is positive and increasing in $x$ if $x^2>d$ and applying Lemma \ref{lemtai},
\begin{align*}
    \|f\|_{p}^p & \le CL^{-1} \sum_{z\in \zc^{(L)} \cap B_R(0)}|f(z)|^pe^{- p|z|^2/2} +C  R^2  e^{-((R-1)^2-d-d\log(\tfrac{(R-1)^2}{d}))/2}\|f\|_{p}^p
    \\ & \le CL^{-1} \sum_{z\in \zc^{(L)} \cap B_R(0)}|f(z)|^pe^{- p|z|^2/2} +C  d (\tfrac{3}{2})^{d/2} e^{-d/4}\|f\|_{p}^p
    \\& \le C^{-1} \sum_{z\in \zc^{(L)} \cap B_R(0)}|f(z)|^pe^{- p|z|^2/2} +\frac{1}{2}\|f\|_{p}^p,
\end{align*}
where the last step holds if $d$ is large enough. Rearranging we get the left-hand side of the Marcinkiewicz-Zygmund inequality \eqref{eqmz}. For the right-hand side notice that by \cite[Theorem 2]{SoTs3} and a union bound, outside a set of probability $O(e^{-c  (\logg d)^{1+2\varepsilon} })$ for every $z\in B_R(0)$,
\[\#\zc^{(L)}\cap B_1(z)= O(L).\] 
Combining this with \eqref{eqloc} completes the proof of \eqref{eqmz} with constants $A_d=AL_d$ and $B_d=BL_d$ where $A,B>0$ are independent of $d$. The final statement follows from a union bound. 
\end{proof}

\begin{rem}
    \label{remrel}
    Proposition \ref{propos} requires oversampling by a factor slighltly bigger than $\sqrt{\logg d}$.
    This can be compared with \cite[Proposition 2]{BaGr3}, which gives a similar estimate for the Paley-Wiener space, when the sampling points are independent and uniformly distributed in an $n$-dimensional cube. Restricting to dimension $2$ and considering a square of area $d$ we see that oversampling by a factor of the order $\logg d$ is needed. In both cases, the oversampling factor corresponds to the scale at which the number of points in balls of unit area remains fairly constant across the region of area $d$ with high probability. For the zeroes of the GEF with $L\gtrsim\sqrt{\logg d}$, this holds because of \cite[Theorem 2]{SoTs3} and a union bound. 
    On the other hand, for the GEF with $L\lesssim (\logg d)^{1/2-\varepsilon}$ this does not hold (see \cite[Theorem 1]{SoTs3}). For the independent samples one can check the right scale through direct computation since the number of points in a ball has a binomial distribution.
\end{rem}

\begin{rem}
    \label{remone} 
If $p=2$, it is possible to get the ratio $B_d/A_d$ in Proposition \ref{propos} arbitrarily close to $1$ if $d_0$ is sufficiently large. This is sometimes referred to as snugness, and entails good numerical properties for the recovery of a function from its samples. This can be achieved by choosing $\Lambda_j$ with density, say, $(\logg d)^{\varepsilon/4}$ rather than constant, carefully tracing the constants and using the fact that lattices of sufficiently high density are snug in the whole Fock space (see \cite{Wa} and also \cite[Theorem 6.5.1]{Gro}).
Naturally, the same should hold for $p\neq 2$, but the authors are unaware of an explicit reference.
\end{rem}

\section{Quantitative estimates for sampling in \texorpdfstring{$\mathcal{F}^p$}{F\^p}}\label{secsa}
In order to prove Theorems \ref{teofull} and \ref{teosam}, we need a quantitative version of \cite[Theorem 1.1]{SeWa}. The proof follows a similar strategy to the original one, but we require two important modifications that make the argument more delicate. Firstly, we have to allow for a set of badly behaved points. Secondly, the global restrictions such as uniform separation between points and their description as a uniformly perturbed lattice are replaced by local restrictions that become less stringent further away from the origin. 

\begin{definition}
    For a discrete set $\zc\in\cc$ and $z\in\zc$ define the product separation at $z\in\zc$ as
    \begin{align*}
    S_\zc(z)&=\prod_{w\in B_1(z)\cap\zc\smallsetminus\{z\}}|z-w|.
    \end{align*}
    We also write $S^{(L)}=S_{\zc^{(L)}}$ for the zero set $\zc^{(L)}$ of the GEF $F_L$. 
\end{definition}
Recalling Definition~\ref{def: s} we see that we may trivially bound $S_\zc(z)\geq s_\zc(z)^{\#B_1(z)\cap\zc}$. However using this crude bound in what follows would lead to significantly worse results. Indeed, in Theorem \ref{teofull} we would need to replace the slowly-growing weight $\omega$ by a function with superpolynomial growth. Accordingly, we work with $S_\zc$.

Throughout this section we work with a set $\zc= \{z_{mn}\}_{m,n\in\zz} \subseteq \cc$ with an associated lattice $\Lambda=\{\lambda_{mn}\}_{m,n\in\zz}=\sqrt{\pi/L}\zz^2$ for some $L>1$ satisfying the following conditions: 
    \begin{enumerate}[label=(\roman*)]
        \item\label{itpl} (Perturbed lattice) For every $m,n\in\zz$, $|z_{mn}-\lambda_{mn}|\le T(|\lambda_{mn}|)$ where $T:[0,\infty)\to [1,\infty)$ is an absolutely continuous function such that $0\le T'(r)\le \tfrac{1}{r}$ for almost every $r>0$;
        \item\label{its} (Separation) For every $z\in\cc$ and $r\ge T(|z|)$,
        \[\#\Big\{w\in\zc\cap B_r(z): \ S_\zc(w)\le S(|{w}|) \Big\}\le \beta \frac{r^2}{\logg^2 r},\]
        where $\beta\ge 1$ and $S:[0,\infty)\to (0,\infty)$ is a decreasing  function such that $-\log S(r)\leq  T(r)^2\logg^4(T(r))$.
    \end{enumerate}

In principle we allow $\zc$ to have repeated entries, since \ref{its} restricts the number of such points.

\begin{proposition}\label{propseip} 
    Let $1\le p<\infty$, $L>1$ and $\Lambda=\{\lambda_{mn}\}_{m,n\in\zz}=\sqrt{\pi/L}\zz^2$. Let $\zc= \{z_{mn}\}_{m,n\in\zz} \subseteq \cc$ be a set satisfying  conditions \ref{itpl} and \ref{its}.
    Then, for every $f\in \mathcal{F}^p$, 
    \begin{align*}
       \|f\|_{p}^p\le C  \sum_{m,n} \omega(m,n) |f(z_{mn})|^p e^{-\frac{p}{2}|z_{mn}|^2},
    \end{align*}
where $\omega(m,n)$ is given by
 \begin{align*}
    \omega(m,n)=\exp\big(C'\beta T(1 +|z_{mn}|)^2\logg^4T(1 +|z_{mn}|)\big),
 \end{align*}
 and the constants $C,C'\ge 1$ depend on $p,L$. Moreover, if $1<L< 2$ one can take $\logg\logg C=\tfrac{C''\beta }{L - 1}$ and $C',C''$ depending only on $p$.
\end{proposition}

As mentioned in Remark \ref{rerem}, the constant $C$ is superfluous whenever an explicit dependence on $L$ is not required.

Let us record here for future reference that
\begin{equation*}
    T(r)-T(1)\leq\int_1^rT'(x)dx\leq \log r\le\logg r \qquad\text{for }r\geq 1
\end{equation*}
which means that $S(r)\geq\exp(-T(r)^2\logg^4(T(r)))$ decays subexponentially.

To prove the proposition we need a version of \cite[Lemma 2.2]{SeWa} with explicit constants. The Weierstrass $\upsigma$-function associated to  $\Lambda$ is defined by
\[\upsigma(z)=z  \sideset{}{'}\prod_{m,n} \big(1-\frac{z}{\lambda_{mn}}\big)e^{\frac{z}{\lambda_{mn}}+\frac{1}{2}\frac{z^2}{\lambda_{mn}^2}},\]
where $\sideset{}{'}\prod$ denotes the product without the term $m=n=0$. Recall from \cite[Sec. 2]{SeWa} that
\begin{align}
    \label{eqsig}
    |e^{-\tfrac{L}{2}|z|^2}\upsigma(z)|\ge c d(z,\Lambda),
\end{align}
for an absolute constant $c>0$.

Let $\zc=\{z_{mn}\}_{m,n\in\zz}$ be a set that satisfies the perturbed lattice condition \ref{itpl} for some $T$. Let us check that $2T(r)=r$ has exactly one solution $r_0\ge 2$. Indeed, since $T(r)\ge 1$, a solution must satisfy $r\ge 2$. Using the fact that $T'(r)\le 1/r$, we see that $u(r)=r-2T(r)$ is continuous and strictly increasing in $(2,+\infty)$. The existence of a unique solution $r_0$ now follows from the fact that $u(2)\le 0$ and $u(r)\to +\infty$ as $r\to +\infty$. In particular, $r\ge 2 T(r)$ if and only if $r\ge r_0$. 

For a decreasing function  $S:[0,\infty)\to (0,\infty)$, define
\begin{align*}
    E&=\Big\{z\in\zc : \ S_\zc(z) \le S(|z|) \Big\},
    \\ A&=\{(m,n)\in\zz^2: \ |\lambda_{mn}|\ge r_0, z_{mn}\notin E\}
    \\&=\{(m,n)\in\zz^2: \ |\lambda_{mn}|\ge 2T( |\lambda_{mn}|), z_{mn}\notin E\},
\end{align*}
and,
\begin{align}
\label{gdef}
    g(z)=\prod_{(m,n)\in A} \big(1-\frac{z}{z_{mn}}\big)e^{\frac{z}{z_{mn}}+\frac{1}{2}\frac{z^2}{\lambda_{mn}^2}}.
\end{align}
The condition \ref{itpl} (and the logarithmic growth of $T$) ensures the convergence of the infinite product. Notice that if $z_{mn}=z_{m'n'}$ for $(m,n)\neq (m',n')$ then $S_{\zc}(z_{mn})=0$ and so $z_{mn}\in E$. In particular, $g$ has simple zeroes. 

The following is a quantitative version of \cite[Lemma 2.2]{SeWa} adapted to our definition of $g$. We postpone the proof to Appendix \ref{applem} since in our case we need to track the constants, which makes the argument quite technical.

\begin{lemma}\label{lemlow}
    With the above notation and under the conditions \ref{itpl} and \ref{its},
    there exist constants $c,C>0$ depending on $L$ such that
    \begin{align}
    \label{eqinf}
        |e^{-\tfrac{L}{2}|z|^2}g(z)|&\ge c d(z,\zc)  \exp\big(-C\beta (T^2(1) \logg^4 T(1)+\tfrac{1}{\logg T(1)}|z|^2)\big).
        \end{align}
    In particular for every $(m,n)\in A$,
    \begin{align}
    \label{eqinf2}
        |e^{-\tfrac{L}{2}|z_{mn}|^2}g'(z_{mn})|\ge  c \exp\big(-C\beta (T^2(1) \logg^4 T(1)+\tfrac{1}{\logg T(1)}|z_{mn}|^2)\big).
    \end{align}
    Also, for $|z|\le \sqrt{2\pi/L}$ we have
    \begin{align}
    \label{eqsup}
        |g(z)|&\le e^{CT(1)}.
    \end{align}
    If $1<L<2$, the dependence on $L$ can be dropped from $c$ and $C$.
\end{lemma}

The next step is to prove an analogous result to \cite[Lemma 3.1]{SeWa} that allows for the set $\zc$ to include bad points and have local restrictions on separation and the perturbed lattice description.

\begin{lemma}
    \label{lemseip}
    Let $\zc\subseteq\cc$ satisfy \ref{itpl} and \ref{its} for some $L$ and $T$. Let $g(z)$ be as in \eqref{gdef}. There exists a constant $C\ge 1$ depending on $L$ such that if $T(1) > C$, then for all $f \in \mathcal{F}^{\infty}$,
    \begin{align*}
        f(z) = \sum_{(m,n) \in A} \frac{f(z_{mn})}{g'(z_{mn})}\frac{g(z)}{z-z_{mn}}.
    \end{align*}
    Moreover, if $L<2$ we can take $\logg C=\tfrac{C'\beta }{L - 1}$, where $\beta $ is from \ref{its} and $C'$ is a sufficiently large (absolute) constant.
\end{lemma}

\begin{proof}
    We prove the statement assuming that $1<L<2$, since otherwise the argument is analogous and no attention has to be paid to tracking the constants' dependence on $L-1$. First we show that the series converges. By \eqref{eqinf2},
    \begin{align*}
    \left|\frac{f(z_{mn})}{g'(z_{mn})}\frac{1}{z-z_{mn}}\right| &\leq |f(z_{mn})e^{-\tfrac{1}{2}|z_{mn}|^2}|\left|\frac{e^{-\frac{L - 1}{2}|z_{mn}|^2+C\beta \big(T(1)^2 \logg^4 T(1) + \tfrac{|z_{mn}|^2}{\logg T(1)}\big)}}{c (z-z_{mn})}\right|
    \\ & \leq \|f\|_{\infty} e^{C\beta T(1)^2\logg^4 T(1)}\frac{e^{-(\frac{L-1}{2} - \frac{C\beta }{\logg T(1)}  )|z_{mn}|^2}}{|z-z_{mn}|},
    \end{align*}
for every $(m,n)\in A$. From the perturbed lattice property (using the fact that $T$ is at most logarithmic) and choosing $T(1)$ large enough to ensure that $ (L-1) > 2C\beta /\logg T(1)$, it follows that the series converges uniformly on compacta to an entire function.
    
    It remains to show that the series coincides with $f(z)$. To this end, let us check that we may take a closed path $\gamma$  lying inside the annulus centred at the origin with inner radius $r$ and outer radius $r+1$, that satisfies $\text{length}(\gamma)\le 8\pi r$ and 
    \[d(\gamma,\zc) > c\min\{S (r+1),(\beta r^{2})^{-1}\},\]
    where $c>0$ is a sufficiently small constant. Here, the factor $S (r+1)$ from the last condition arises from ensuring that $\gamma$ can stay clear from the  ``well'' separated zeroes (which are most of them), while $(\beta r^{2})^{-1}$ allows $\gamma$ to avoid the ``badly'' separated points (of which there are much fewer). To see this, we assume $r$ is sufficiently large and use \ref{its} for $z=0$ and radius $r+1$. First notice that since there are $O(\beta  \tfrac{r^2}{\logg^2 r})$ badly separated points in the annulus $B_{r+1}(0)\smallsetminus B_r(0)$, we can choose a radius $r+\tfrac{1}{4}\le r_{1}\le r+\tfrac{3}{4}$ such that these points are at a distance at least $c(\beta r^{2})^{-1}$ from the circle of centre $0$ and radius $r_{1}$. On the other hand, if $z\in\zc\cap (B_{r+1}(0)-B_r(0))$ is a well separated point, then $s_\zc(z)\ge S_{\zc}(z)\ge S(r+1)$. For $r_{2}=\min\{\tfrac{S(r+1)}{2},\tfrac{1}{4}\}$ it follows that every point of $\zc$ is at a distance at least $r_2$ from the circle of centre $z$ and radius $r_{2}$. So, it suffices to take $\gamma$ as the curve that mainly follows the circle of radius $r_1$, but avoids well separated points $z$ close to this circle by following an arc of the circle centred at $z$ and of radius $r_{2}$. Since $r_{2}\le\tfrac{1}{4}$, we see that $\gamma$ is bounded by the circles centred at $0$ of radii $r$ and $r+1$. 
    
    Denoting the domain enclosed by $\gamma$ by $\Omega$, for $z \in \Omega \smallsetminus \zc$ we obtain
    \begin{align*}
        \frac{1}{2\pi i}\int_{\gamma} \frac{f(\zeta)}{g(\zeta)(\zeta - z)}d\zeta = \frac{f(z)}{g(z)} - \sum_{z_{mn} \in \Omega} \frac{f(z_{mn})}{g'(z_{mn})}\frac{1}{z-z_{mn}}.
    \end{align*}
    By \eqref{eqinf} and the fact that  $S$ decays subexponentially we have
    \begin{align*}
        \left|\int_{\gamma} \frac{f(\zeta)}{g(\zeta)(\zeta - z)}d\zeta\right|&\leq   \frac{C r \|f\|_{\infty}e^{C\beta  T(1)^2 \logg^ 4 T(1)} e^{-(\frac{L-1}{2} -\frac{C\beta }{\logg T(1)})r^2}}{d(\gamma, z)\min\{S (r+1),(\beta r^{2})^{-1}\}}\\
        &\leq  \|f\|_{\infty}e^{C\beta  T(1)^2 \logg^4 T(1)}\frac{e^{-(\frac{L-1}{2} -\frac{C\beta }{\logg T(1)})r^2/2}}{d(\gamma, z)}.
    \end{align*}
    The proof concludes by noting that for fixed $z$ the right hand side tends to 0 as $r \to \infty$.
\end{proof}

Having established the necessary groundwork, we now prove Proposition \ref{propseip}.

\begin{proof}[Proof of Proposition \ref{propseip}]
 As before, we assume that $1<L<2$, since otherwise the argument is analogous and no attention has to be paid to tracking the constants' dependence on $L-1$.
We also assume for now that $\logg T(0) > \frac{C\beta }{L - 1}$ for a sufficiently large constant $C$ depending on $p$. 

Let $\Delta=\sqrt{\pi/L}[0,1)^2$ be the fundamental domain of the lattice $\Lambda$ and recall the Bargmann-Fock shift $\mathcal{T}_{a}$ from \eqref{eqshi}. With this notation $\|f\|_{p}^p$ can be computed as
\begin{align}\label{eqper}
    \frac{p}{2\pi}\int_{\cc} |f(z)|^p e^{-\frac{p}{2}|z|^2} dA(z) = \frac{p}{2\pi}\sum_{(k,l)\in \zz^2} \int_\Delta|\mathcal{T}_{\lambda_{kl}}f(z)|^p e^{-\frac{p}{2}|z|^2}dA(z).
\end{align}
We now consider for each pair $(k,l)\in\zz^2$ the sets 
\begin{align*}
    \zc_{kl} &= \{z^{(kl)}_{mn}\}_{m,n} = \{z_{m-k,n-l}+ \lambda_{kl}\}_{m,n},
    \\ E_{kl} &= \{z \in \zc_{kl} \colon S_{\zc_{kl}}(z) \leq S(|z| + |\lambda_{kl}|) \}, \quad \text{and,}
   \\  A_{kl}& = \{(m,n)\in \zc^2 \colon |\lambda_{mn}| \geq 2T(|\lambda_{mn}| + |\lambda_{kl}|), z^{(kl)}_{mn} \not \in E_{kl} \}.
\end{align*}
Let $g_{\lambda_{kl}}(z)$ be the function defined in \eqref{gdef} for the corresponding set $\zc_{kl}$, that is,
\begin{align*}
    g_{\lambda_{kl}}(z)=\prod_{(m,n)\in A_{kl}} \Big(1-\frac{z}{z^{(kl)}_{mn}}\Big)\exp\Big(\frac{z}{z^{(kl)}_{mn}}+\frac{1}{2}\frac{z^2}{\lambda_{mn}^2}\Big).
\end{align*}

A straightforward but tedious computation shows that the sets $\zc_{kl}$ satisfy the conditions \ref{itpl} and \ref{its} with the corresponding functions $T_{kl}(r) = T( r + |\lambda_{kl}|)$ and $S_{kl}(r) = S(r + |\lambda_{kl}|)$. In particular, the hypotheses of Lemmas \ref{lemlow} and \ref{lemseip} are satisfied.

Thus we can write
\begin{align*}
    \mathcal{T}_{\lambda_{kl}}f(z) = \sum_{(m,n) \in A_{kl}}\frac{\mathcal{T}_{\lambda_{kl}}f(z^{(kl)}_{mn})}{g'_{\lambda_{kl}}(z^{(kl)}_{mn})}\frac{g_{\lambda_{kl}}(z)}{z-z^{(kl)}_{mn}},
\end{align*}
Hölder's inequality $(1/p + 1/q = 1)$ then yields
\begin{align*}
    |\mathcal{T}_{\lambda_{kl}}f(z)| \leq I_{kl}^{1/p} h_{kl}^{1/q}(z),
\end{align*}
where 
\begin{align*}
    I_{kl} &=  \sum_{A_{kl}}\frac{|\mathcal{T}_{\lambda_{kl}}f(z^{(kl)}_{mn})|^p}{|g'_{\lambda_{kl}}(z^{(kl)}_{mn})|^p } e^{p\frac{L-1}{4} |z^{(kl)}_{mn}|^2},\quad\text{and,}
    \\ h_{kl}(z) &= \sum_{A_{kl}}e^{-q\tfrac{L-1}{4}|z^{(kl)}_{mn}|^2}\left|\frac{g_{\lambda_{kl}}(z)}{z-z^{(kl)}_{mn}}\right|^q.
\end{align*}

In consequence,
\begin{align}
\label{eqtrans}
    \int_\Delta |\mathcal{T}_{\lambda_{kl}}f(z)|^pe^{-\frac{p}{2}|z|^2}dA(z) \leq I_{kl}\int_\Delta h_{kl}^{p/q}(z) e^{-\frac{p}{2}|z|^2}dA(z).  
\end{align}
Note that if $(m,n) \in A_{kl}$, 
    \begin{align*}
        |\lambda_{mn}|\le |\lambda_{mn}-z^{(kl)}_{mn}| + |z^{(kl)}_{mn}|\le T_{kl}( |\lambda_{mn}|) + |z^{(kl)}_{mn}|\le \tfrac{|\lambda_{mn}|}{2} + |z^{(kl)}_{mn}|.
    \end{align*}
    In particular, 
    \begin{align}
    \label{eqlamt2}
        |z^{(kl)}_{mn}|\ge \tfrac{|\lambda_{mn}|}{2} \ge T_{kl}( |\lambda_{mn}|).
    \end{align}
So, for $z\in\Delta$, we have
\[|z-z^{(kl)}_{mn}| \geq  T_{kl}(|\lambda_{mn}|) - \sqrt{2\pi}\ge T(0) - \sqrt{2\pi}\ge 1,\] 
where the last step follows from our assumption $\logg T(0)\ge \tfrac{C\beta }{L-1}$, provided $C$ is sufficiently large. Plugging this into \eqref{eqtrans} and using \eqref{eqsup} and \eqref{eqlamt2} we get
\begin{align}\label{eqtau}
    \int_\Delta |\mathcal{T}_{\lambda_{kl}}f(z)|^pe^{-\frac{p}{2}|z|^2}dA(z) 
    & \leq  Ce^{CT_{kl}(1)} I_{kl}\Big(\sum_{A_{kl}}e^{-q\tfrac{L-1}{4}|z^{(kl)}_{mn}|^2}\Big)^{p/q}
    \\ &\leq I_{kl} Ce^{CT_{kl}(1)} \Big(\sum_{A_{kl}}e^{-q\tfrac{L-1}{16}|\lambda_{mn}|^2}\Big)^{p/q} \notag
    \\ &\leq I_{kl} \frac{Ce^{CT_{kl}(1)} }{(L-1)^{p/q}}\leq I_{kl} C\left(\logg T(0)\right)^{p/q}e^{CT_{kl}(1)}\notag\\ &\le I_{kl} C\left(\logg T_{kl}(1)\right)^{p/q}e^{CT_{kl}(1)} \le I_{kl} e^{CT_{kl}(1)},\notag
\end{align}
where in the penultimate step we used again that $\logg T(0)\ge \tfrac{C\beta }{L-1}\ge \tfrac{1}{L-1}$. Regarding $I_{kl}$, \eqref{eqinf2} and a straightforward but lengthy computation give
\begin{align*}
   I_{kl} &
   \leq e^{Cp\beta T_{kl}(1)^2 \logg^4 T_{kl}(1)} 
   \sum_{A_{kl}}|f(z_{m-k,n-l})|^pe^{-\frac{p}{2}|z_{m-k,n-l}|^{2}} e^{-p\big(\frac{L-1}{4} - \tfrac{C\beta }{\logg T_{kl}(1)}\big)|z^{(kl)}_{mn}|^2}
\\ &\leq e^{C\beta T_{kl}(1)^2 \logg^4 T_{kl}(1)}     \sum_{A_{kl}}|f(z_{m-k,n-l})|^pe^{-\frac{p}{2}|z_{m-k,n-l}|^{2}} e^{-p\frac{L-1}{32} |\lambda_{mn}|^2},
\end{align*}
where in the last step we used \eqref{eqlamt2} and that $T(0)$ is sufficiently large.

Returning to our computation of $\|f\|_{p}^p$, using this bound, \eqref{eqper} and \eqref{eqtau}, we get 
\begin{equation}
    \label{efbound}
\begin{split}
    \|f\|_{p}^p &\leq  
        \sum_{(k,l)\in \zz^2} 
        \begin{multlined}[t]e^{C\beta T_{kl}(1)^2 \logg^4 T_{kl}(1)} \cdot
    \\ \cdot\sum_{A_{kl}} |f(z_{m-k,n-l})|^p e^{-\frac{p}{2}|z_{m-k,n-l}|^2}e^{-\frac{p(L-1)}{32}|\lambda_{mn}|^2}
    \end{multlined}
    \\ &\leq  
        \sum_{(m,n)\in \zz^2}  \begin{multlined}[t] |f(z_{mn})|^p e^{-\frac{p}{2}|z_{mn}|^2}
        \cdot
    \\ \cdot \sum_{(k,l)\in \zz^2} e^{C\beta T_{kl}(1)^2 \logg^4 T_{kl}(1)} e^{-\frac{p(L-1)}{32}|\lambda_{mn}+\lambda_{kl}|^2}
    \end{multlined}
    \\ &\leq  
        \sum_{(m,n)\in \zz^2}  \begin{multlined}[t] |f(z_{mn})|^p e^{-\frac{p}{2}|z_{mn}|^2}
        \cdot
    \\ \cdot \sum_{(k,l)\in \zz^2} e^{C\beta T(1 + |\lambda_{kl}| +|\lambda_{mn}|)^2 \logg^4 T(1 + |\lambda_{kl}|+|\lambda_{mn}|)} e^{-\frac{p(L-1)}{32}|\lambda_{kl}|^2}.
    \end{multlined}
\end{split}
\end{equation}

It remains to bound the last sum on the right-hand side. Distinguishing between two cases according to whether $|\lambda_{mn}|$ is larger than $|\lambda_{kl}|$ or not, and  using that $T(y)-T(x)\le \logg(y/x)$ for $0<x<y$, a straightforward computation shows that
\[T(1 + |\lambda_{kl}| +|\lambda_{mn}|)^2 \logg^4 T(1 + |\lambda_{kl}|+|\lambda_{mn}|)\le C \big(T(1  +|\lambda_{mn}|)^2 \logg^4 T(1 +|\lambda_{mn}|)+|\lambda_{kl}|\big).\]
In particular,
\begin{align}
\label{eqlas}
     \sum_{(k,l)\in \zz^2}& e^{C\beta T(1 + |\lambda_{kl}| +|\lambda_{mn}|)^2 \logg^4 T(1 + |\lambda_{kl}|+|\lambda_{mn}|)} e^{-\frac{p(L-1)}{32}|\lambda_{kl}|^2}
    \\ &\le  e^{C\beta T(1 + |\lambda_{mn}|)^2 \logg^4 T(1 + |\lambda_{mn}|)}
     \sum_{(k,l)\in \zz^2}  e^{C\beta |\lambda_{kl}|-\frac{p(L-1)}{32}|\lambda_{kl}|^2}\notag
     \\&\le  e^{C\beta T(1 + |\lambda_{mn}|)^2 \logg^4 T(1 + |\lambda_{mn}|)}e^{\tfrac{C\beta^2}{L-1}}\le e^{C\beta T(1 + |\lambda_{mn}|)^2 \logg^4 T(1 + |\lambda_{mn}|)},\notag
\end{align}
where the penultimate step follows from separating the sum according to whether $C\beta \le \tfrac{p(L-1)}{64}|\lambda_{kl}|$ or not. Notice that since $|\lambda_{mn}|\le|z_{mn}|+T(|\lambda_{mn}|)$ we have
\begin{align*}
    T(1+|\lambda_{mn}|) 
    &\le T(1+|z_{mn}|)+\log \Big(\frac{1+|z_{mn}|+T(|\lambda_{mn}|)}{1+|z_{mn}|}\Big)
    \\&\le T(1+|z_{mn}|)+\log(2T(|\lambda_{mn}|))\le T(1+|z_{mn}|)+\frac{2T(1+|\lambda_{mn}|)}{e}.
\end{align*}
So, we can replace $\lambda_{mn}$ by $z_{mn}$ on the right-hand side of \eqref{eqlas}. This, together with \eqref{efbound}, completes the proof, provided that $\logg T(0) > \frac{C\beta }{L-1}$ for a sufficiently large constant.

If this is not the case, we can take $\widetilde{T}(r) = T(r) + e^{\frac{C\beta }{L-1}}$ which verifies the same conditions as $T$. Applying the previous argument, we obtain 
 \begin{align*}
       \|f\|_{p}^p\le   \sum_{m,n} \omega(m,n) |f(z_{mn})|^p e^{-\frac{p}{2}|z_{mn}|^2},
    \end{align*}
where $\omega(m,n)$ is given by
 \begin{align*}
    \omega(m,n)=\exp\Big(C'\beta \big(T(1 +|z_{mn}|)+e^{\frac{C\beta }{L-1}}\big)^2\logg^4\big(T(1 +|z_{mn}|)+e^{\frac{C\beta }{L-1}}\big)\Big).
 \end{align*}
 The proof follows again by a simple but lengthy computation showing that
 \[
    \omega(m,n)\le \exp(\exp(\tfrac{C\beta }{L-1}))\exp\big(C'\beta (T(1 +|z_{mn}|))^2\logg^4(T(1 +|z_{mn}|))\big).\qedhere
 \]
\end{proof}

\section{Separation for the GEF}\label{secse}
In this section we study the separation between the zeroes of the  GEF. The estimates obtained will be used in Lemma \ref{lemprob} to show that the zeroes satisfy the perturbed lattice and separation conditions \ref{itpl} and \ref{its} with high probability. We start by comparing the $k$-point intensity of the GEF with that of the determinantal hyperbolic GAF. In essence we need a result like \cite[Theorem 1.1]{NaSo3}, but with explicit constants.

\begin{rem}\label{remop}
    Let $F(z)=\sum a_n \zeta_n z^n$ and $\widetilde{F}(z)=\sum b_n \zeta_n z^n$ be two GAFs such that $0\le a_n\le b_n$. Then $\Gamma_F\le \Gamma_{\widetilde{F}}$. Indeed, notice that
    \begin{align*}
     \Gamma_F=\begin{pNiceArray}{c|c}
  \Big(\sum_{n\ge 0} a_n^2 (z_i\overline{z_j})^n\Big)_{ij} & \Big(\sum_{n\ge 0} a_n^2 n z_i^n\overline{z_j}^{n-1}\Big)_{ij} \\
  \hline
   \Big(\sum_{n\ge 0} a_n^2 n z_i^{n-1}\overline{z_j}^{n}\Big)_{ij} & 
   \Big(\sum_{n\ge 0} a_n^2 n^2 (z_i\overline{z_j})^{n-1}\Big)_{ij}
\end{pNiceArray}
    =MDM^*,
    \end{align*}
    where $M$ and $D$ are (unbounded) operators from $\ell_2$ to $\cc^{2k}$ and $\ell_2$ respectively, $D$ is diagonal with $D_{nn}=a_{n-1}^2$, and 
    \begin{align*}
        M_{i,n}=\begin{cases}
            z_i^{n-1} & \text{if } \hphantom{n+1} 1\le i \le k
          \\ (n-1) z_i^{n-2} & \text{if }   \hphantom{1} k+1\le i \le 2k
        \end{cases}.
    \end{align*}

    Since $M D_1 M^*\le M D_2 M^*$ whenever $D_1\le D_2$ (in the sense of operators), the claim follows.
\end{rem}

\begin{lemma}\label{lemrho}
    Let $k\in\nn$, $\tau\ge 2$ and $F(z)=\sum a_n \zeta_n z^n$ be a GAF.
    Then, for $z_1,\ldots,z_k\in B_\tau(0)$,
    \begin{align*}
        \rho_{k,F}(z_1,\ldots,z_k)
        \le \Big(\frac{\sup_{n\ge 0} a_n (2\tau)^{n}}{\min_{0\le n<2k} a_n}\Big)^{4k} \prod_{i<j}|z_i-z_j|^2.
    \end{align*}
    In particular, if $F$ is the GEF of intensity $L/\pi$,
    \begin{align}\label{eqint}
        \rho_{k,F}(z_1,\ldots,z_k)
        \le (2k)^{4k^2}e^{8\tau^2L k}\max\{1,L^{-4k^2}\} \prod_{i<j}|z_i-z_j|^2.
    \end{align}
\end{lemma}
    
\begin{proof}
    Recall from \eqref{eqrho} that for $z_1,\ldots,z_k\in B_\tau(0)$,
    we have that
    \begin{align*}
    \rho_F(z_1\ldots,z_k)=\frac{1}{\pi^{2k}\det (\Gamma_F)}\int_{\cc^k} |\eta_1\ldots\eta_k|^2 e^{-\tfrac{1}{2}\langle \Gamma_F^{-1}\eta', \eta'\rangle}d\eta_1\ldots d\eta_k,
    \end{align*}
    where $\eta'=(0,\eta)$.

    First we apply Remark \ref{remop} to estimate $\det(\Gamma_F)$ from below. 
    Let $a=\min_{0\le n<2k} a_n$ and assume without loss of generality that $a>0$. We compare $F$ with the GAF $ G=a\sum_{n=0}^{2k-1} \zeta_n z^n$. From \eqref{eqmat}, we get
   \[\Gamma_F \ge \Gamma_G=a^2 M_V M_V^*.\]
    In particular, by Lemma \ref{lemcon},
    \begin{align}\label{eqdet2}
        \det(\Gamma_F)\ge a^{4k} \prod_{i<j}|z_i-z_j|^8.
    \end{align}

Regarding the integral in \eqref{eqrho}, we use Remark \ref{remop} again to compare $F$ with the determinantal GAF $\widetilde G(z)=\sum_n \zeta_n (\tfrac{z}{2\tau})^n$. Let ${b}=\sup_{n\ge 0} a_n (2\tau)^{n}$ and suppose without loss of generality that ${b}<\infty$. With this,
    \[\Gamma_F\le {b}^2 \Gamma_{\widetilde{G}}.\]
    So we have
    \begin{align*}
        \int_{\cc^k} |\eta_1\ldots\eta_k|^2 e^{-\tfrac{1}{2}\langle \Gamma_F^{-1}\eta', \eta'\rangle}dA^k(\eta_1\ldots \eta_k) 
    &\le \int_{\cc^k} |\eta_1\ldots\eta_k|^2 e^{-\tfrac{1}{2}\langle {b}^{-2}\Gamma_{\widetilde G}^{-1}\eta', \eta'\rangle}dA^k(\eta_1\ldots \eta_k)
    \\ &\le {b}^{4k} \int_{\cc^k} |\eta_1\ldots\eta_k|^2 e^{-\tfrac{1}{2}\langle \Gamma_{\widetilde G}^{-1}\eta', \eta'\rangle}dA^k(\eta_1\ldots \eta_k)
    \\ &=  {b}^{4k}\pi^{2k}\det(\Gamma_{\widetilde G}) \rho_{\widetilde G}(z_1,\ldots,z_k).
    \end{align*}
    Combining this with \eqref{eqrho}  and \eqref{eqdet2} we get
    \begin{align}\label{eqrhos}
        \rho_F(z_1\ldots,z_k)\le \frac{{b}^{4k}}{a^{4k}\prod_{i<j}|z_i-z_j|^8} \det(\Gamma_{\widetilde G}) \rho_{\widetilde G}(z_1,\ldots,z_k).
    \end{align}
    It remains to estimate $\det (\Gamma_{\widetilde G})$. 
Rescaling \eqref{eqmat} and applying Lemma \ref{lemcon} we arrive at
\[\det (\Gamma_{\widetilde G})=(2\tau)^{2k}\frac{\prod_{1\le i<j\le k} |z_j-z_i|^8}{\prod_{i,j=1}^k (4\tau^2-z_i\overline{z_j})^4}\le \prod_{1 \le i<j\le k} |z_j-z_i|^8, \]
where in the last step we used that $z_1,\ldots,z_k\in B_\tau(0)$. This together with \eqref{eqrhos} gives
\begin{align}\label{eqextr}
    \rho_{F}(z_1,\ldots,z_k)
        \le \big(\tfrac{{b}}{a}\big)^{4k} \rho_{\widetilde G}(z_1,\ldots,z_k).
\end{align}
    By \cite[Theorem 1]{PeVi} and Borchardt's permanent-determinant identity from \cite{Bro} (see also \cite[Proposition 5.1.5]{HKPV}), 
    \begin{align*}
        \rho_{\widetilde G}(z_1,\ldots,z_k)&= (4\pi\tau^2)^{-k}\det \big(\big((1-z_i\overline{z_j}/4\tau^2)^{-2}\big)_{ij}\big) 
        \\ &=(4\pi\tau^2)^{-k}\det\big(\big((1-z_i\overline{z_j}/4\tau^2)^{-1}\big)_{ij}\big)\perm\big(\big((1-z_i\overline{z_j}/4\tau^2)^{-1}\big)_{ij}\big),
    \end{align*}
    where $\perm$ is the permanent of a matrix. Applying Cauchy's determinant formula and a straightforward bound for the permanent shows that for $z_1,\ldots,z_k\in B_\tau(0)$,
    \[\rho_{\widetilde G}(z_1,\ldots,z_k) \le \prod_{i<j}|z_i-z_j|^2.\]
    This together with \eqref{eqextr} proves the general bound. In particular, \eqref{eqint} follows from the fact that $(4\tau^2L)^{n}/n!\le e^{4\tau^2L}$ for every $n\ge 0$ and $L^n/n!\ge \min\{1,L^{2k}\}(2k)^{-2k}$ for every $0\le n < 2k$.
\end{proof}

\begin{proposition}
    [Small scale separation estimate]
    \label{coroloca} 
There are absolute constants $c,C>0$ such that for $L>0$, $\tau\ge 2$, $\sigma>0$ and $k\in\nn$,
\begin{align*}
     P(\#\{z&\in B_\tau(0)\cap \zc^{(L)}: \ S^{(L)}(z)<\sigma\}\ge k)\le e^{-cL^2 \tau^4} +e^{CL^2 \tau^4\logg (\max\{L,L^{-1}\} \tau)} \sigma^{k}.
\end{align*}
\end{proposition}

\begin{proof} 
    Rescaling \cite[Theorem 2]{SoTs3} for intensity $L/\pi$, we have that except for an event of probability $e^{-cL^2\tau^4}$, there are at most $2L(\tau+1)^2$ zeroes in $B_{\tau+1}(0)$. Assume that $B_{\tau+1}(0)\cap \zc^{(L)}=\{z_i\}_{i=1}^\ell$ and that there are at least $k$ zeroes $z_{i_j}\in B_{\tau}(0)$ such that
    \[S^{(L)}(z_{i_j})=\prod_{w\in B_1(z_{i_j})\cap\zc^{(L)}\smallsetminus\{z_{i_j}\}}|z_{i_j}-w|<\sigma,\quad 1\le j \le k.\]
    Then,
    \[\varrho_\ell(z)=\varrho_\ell(z_1,\ldots,z_\ell):= \prod_{i<j}|z_i-z_j|^2\le (2\tau+2)^{\ell^2} \sigma^k.\]
    So,
    \begin{align}\label{eqpart}
        P(\#\{z&\in B_\tau(0)\cap \zc^{(L)}: \ \prod_{w\in B_1(z)\cap\zc^{(L)}\smallsetminus\{z\}}|z-w|<\sigma\}\ge k)
        \\ &\le e^{-cL^2\tau^4} + \sum_{\ell=k}^{2L(\tau+1)^2} P(B_{\tau+1}(0)\cap\zc^{(L)}=\{z_i\}_{i=1}^\ell : \ \varrho_\ell(z)\le (2\tau+2)^{\ell^2} \sigma^k)  \notag
    \end{align}
    By Lemma \ref{lemrho},
    \begin{align*}
        P(B_{\tau+1}(0)&\cap\zc^{(L)}=\{z_i\}_{i=1}^\ell : \ \varrho_\ell(z)\le (2\tau+2)^{\ell^2} \sigma^k) 
        \\ &\le \ee\Big[\sideset{}{^{\neq}}\sum_{z_1,\ldots,z_\ell\in\zc^{(L)}} \one_{\{\varrho_\ell(z)\le(2\tau+2)^{\ell^2} \sigma^k\}} \prod_i \one_{B_{\tau+1}(0)}(z_i) \Big]
        \\ &= \int_{B_{\tau+1}(0)^{\ell}\cap \{\varrho_\ell(z)\le(2\tau+2)^{\ell^2} \sigma^k\}} \rho_{\ell,F}(z)  dA^\ell(z_1\ldots z_{\ell})
        \\ &\le (2\ell)^{4\ell^2} e^{8(\tau+1)^2L \ell} \max\{1,L^{-4\ell^2}\} \int_{B_{\tau+1}(0)^{\ell}\cap \{\varrho_\ell(z)\le(2\tau+2)^{\ell^2} \sigma^k\}} \varrho_\ell(z)  dA^\ell(z_1\ldots z_{\ell})
        \\ &\le C^{\ell^2}\ell^{4\ell^2}  e^{8(\tau+1)^2L \ell} \max\{1,L^{-4\ell^2}\} (\tau+1)^{2\ell}(2\tau+2)^{\ell^2}\sigma^{k}
        \\ &\le   e^{CL^2 \tau^4\logg (\max\{L,L^{-1}\} \tau)} \sigma^{k}.
    \end{align*}
    Combining this with \eqref{eqpart} we conclude that
    \begin{align*}
        P(\#\{z&\in B_\tau(0)\cap \zc^{(L)}: \ \prod_{w\in B_1(z)\cap\zc^{(L)}\smallsetminus\{z\}}|z-w|<\sigma\}>k)
        \\ &\le e^{-cL^2 \tau^4} + e^{CL^2 \tau^4\logg (\max\{L,L^{-1}\} \tau)} \sigma^{k}. \qedhere
    \end{align*}
\end{proof}

\begin{proposition}[Large scale separation estimate] \label{lemindep}
    For $\tau>\tfrac{1}{\sqrt{L}}$ and $N\in\nn$, let $\mathcal{Q}$ be a collection of $N$ disjoint translations of $[0,\tau)^2$. There are absolute constants $c,C,D>0$ such that for $0<\varsigma<\tfrac{1}{\sqrt{L}}$ and 
    \[\lambda\ge D N (L^3 \tau^2 \varsigma^4+L \tau^3\varsigma^{-1}e^{-L \tau^2}),\] 
    we have 
    \begin{align*}
         P(\#\{Q\in\mathcal{Q}:\ \exists z\in\zc^{(L)}\cap Q, \ s^{(L)}(z)<\varsigma\}>  \lambda )
        \le e^{-c\lambda} + e^C N e^{-e^{L\tau^2}}.
    \end{align*} 
\end{proposition}
\begin{proof}
    By a rescaling argument, we can assume that $L=1$. Given a numerical constant $A>1$ we can assume without loss of generality that the squares in $\qc$ are $A\tau$-separated. This can be achieved by splitting $\qc$ into a finite number (depending only on $A$) of disjoint families and proving the estimate for each one of them. Let us also introduce the notation $F^*(z)=e^{-\tfrac{1}{2}|z|^2}F(z)$. From \cite[Theorem 3.1]{NaSo}, there is a constant $C>0$ such that  $F^*=F_j^*+G_j^*$ on  $\widetilde{Q}_j:=Q_j+[-2,2]^2$ for each $Q_j\in\qc$, where $F_j$ are independent copies of the GEF and
    \[\logg P(\max_j\max_{z\in \widetilde{Q}_j}|G_j^*(z)|\ge e^{-\tau^2})\le C+\logg N -e^{\tau^2}.\]
    Fix a square $Q_j\in\qc$. Let $\{w_k\}_{k}$ be an $\varsigma$-net for $Q_j$ of size at most $\lceil\tau/\varsigma\rceil^2$ and let $\gamma$ be the curve given by $\bigcup_k \partial B_{2\varsigma}(w_k)\subseteq \widetilde{Q}_j$. By \cite[Lemma 8]{NaSoVo}, 
    \begin{align}
        \label{eqlazo}
        P(E_j):=P(\min_{z\in\gamma} |F_j^*(z)|< e^{-\tau^2})\le C \frac{\tau^3}{\varsigma}e^{-\tau^2}.
    \end{align}
    
    Now assume that $\max_j\max_{z\in \widetilde{Q}_j}|G_j^*(z)|< e^{-\tau^2}$. If there exists $z\in\zc^{(L)}\cap Q_j$ such that $s^{(L)}(z)<\varsigma$, then either the event $E_j$ happens, or, we can apply Rouché's Theorem to ensure that $F_j$ and $F$ have the same number of zeroes in $B_{2\varsigma}(w_k)$ for every $k$. In this case, $F_j$ will have at least two zeroes in $B_{2\varsigma}(w_k)$ where $k$ is such that $z\in B_\varsigma(w_k)$. In particular, there exists $w\in \widetilde{Q}_j$ such that  $s_{F_j}(w)<2\varsigma$. We call this event $\Omega_j$. We have shown that
    \begin{multline}
    \label{eqprobo}
        P(\#\{Q\in\mathcal{Q}:\ \exists z\in\zc^{(L)}\cap Q, \ s^{(L)}(z)<\varsigma\}> \lambda)
    \\ \le P\Big(\sum_j 1_{\Omega_j\cup E_j} >  \lambda \Big)+ e^C N e^{-e^{\tau^2}}.
    \end{multline}
    
    Note that the $\{0,1\}$-valued variables $1_{\Omega_j\cup E_j}$ are  independent since they are defined in terms of $F_j$. Also, by \eqref{eqlazo} and Lemma \ref{lemloc},
    \[\ee\Big[\sum_j 1_{\Omega_j\cup E_j}\Big]\le \sum_{j=1}^N P(\Omega_j)+ P(E_j)\le CN(\tau^2\varsigma^4+\tau^3\varsigma^{-1}e^{-\tau^2})\le \tfrac{\lambda}{2},\]
    assuming (as we may) that $2C\le D$.
    Applying Bernstein's inequality to the variables $X_j=1_{\Omega_j\cup E_j}-\ee[1_{\Omega_j\cup E_j}]$,
    \begin{align*}
         P\Big(\sum_j 1_{\Omega_j\cup E_j} > \lambda\Big)
         \le P\Big(\sum_j X_j >  \tfrac{\lambda}{2}\Big)
         \le e^{-c\lambda}.
    \end{align*}
    Combining this with \eqref{eqprobo} completes the proof.
\end{proof}

\section{Local sampling for the GEF}
\label{secGEF}

In this section we use the separation estimates of the previous section to show that the zeroes of the GEF satisfy the perturbed lattice and separation conditions \ref{itpl} and \ref{its} with high probability. With this we prove our main results Theorems \ref{teofull} and \ref{teosam}.

Let $F$ be the GEF of intensity $L/\pi$. We say that an event depending on $R>0$ has negligible probability in terms of $R$, if its probability is $\le C_L\exp(-c_L\logg R\logg\logg R)$. We also denote:
\begin{itemize}
    \item $T(r):=\logg^{1/4}r\logg\logg r$
    \item $S(r):=\exp(-T(r)^2\logg^4 T(r))$
    \item $\kappa(r):=r^2/\logg^2 r$
\end{itemize}

The following lemma shows that the zeroes of the GEF satisfy \ref{itpl} and \ref{its} with high probability.

\begin{lemma}
\label{lemprob}
Let $L>0$. Outside a set of negligible probability in terms of $R$, the following statements hold simultaneously:
    \begin{enumerate}[label=(\Roman*)]
        \item 
        \label{rigid2} 
        (Perturbed lattice) For every $z\in\zc^{(L)}$ and $\lambda\in\Lambda=\sqrt{\pi/L} \zz^2$ its associated lattice point,  
        \[|z-\lambda|\le T(|\lambda|+R) \le T(2|z|+2R);\]
        \item 
        \label{rigiid2} 
        (Separation) There is a constant $C_L>0$ depending on $L$ such that for every $z\in\cc$ and $r\ge T(e|z|+R)$,
        \[\#\{w\in\zc^{(L)}\cap B_r(z): \ S^{(L)}(w)\le S(2e|w|+R) \}\le C_L \kappa(r).\]
    \end{enumerate}
Moreover, if ${L_0}<L<2{L_0}$ for some ${L_0}>0$, then one can choose the implied constants dependent on ${L_0}$ rather than $L$.
\end{lemma}
\begin{proof}
 We assume that ${L_0}<L<2{L_0}$ for some ${L_0}$ and track the dependence of the constants in terms of ${L_0}$ rather than $L$. To then obtain constants depending on $L$ we can just take ${L_0}=\tfrac{2}{3}L$.
 
We can also assume that $R\ge R_{L_0}$ for some sufficiently large constant $R_{L_0}$ depending on ${L_0}$ (otherwise we can set the failure probability to 1 so that the stament trivially holds). 

Rescaling \eqref{eqdev},
    \begin{align*}
        P(|z_{00}|>r)\le Ce^{-cL^2 r^4/\logg (\sqrt{L}r)}\le Ce^{-c{L_0}^2 r^4/\logg r},\quad r>\sqrt{{L_0}}.
    \end{align*}
    For $k\ge -1$ define
    \[R_k=Re^k,\quad\text{and}\quad T_k=T(R_k).\]
    By a union bound, the translation invariance of $\zc^{(L)}$ and assuming that $R_{L_0}$ is sufficiently large (specifically, we require $T(R_{L_0}/e)>\sqrt{L_0}$ and $\logg\logg R_{L_0}\ge {L_0}^{-1}$), we get
    \begin{align*}
        P(\exists k\in \nn_0,&\exists m,n\in \zz: \ \lambda_{mn}\in B_{R_k}(0), |z_{mn}-\lambda_{mn}|>T_{k-1})
        \\ &\le C \sum_{k=0}^\infty R_k^2 P(|z_{00}|>T_{k-1})
        \le C \sum_{k=0}^\infty R_k^2  e^{-c{L_0}^2 T_{k-1}^4/\logg T_{k}}
        \\ &\le C R^2  \sum_{k=0}^\infty e^{2k - c{L_0}^2\logg R_k\logg^3\logg R_k} 
        \le C R^2  \sum_{k=0}^\infty e^{2k - c\logg R_k\logg\logg R_k} 
        \\ &\le C R^2  \sum_{k=0}^\infty e^{2k - c(\logg R+k)\logg\logg R}
        \le C R^2  e^{-c\logg R\logg\logg R}\sum_{k=0}^\infty e^{-k(c\logg\logg R-2)}
        \\ &\le  C  e^{-c\logg R\logg\logg R}.
    \end{align*}
    So, the first inequality from \ref{rigid2} holds outside an event of negligible probability. The second inequality follows from the fact that
    \[|\lambda|+R\le |z|+R+T(|\lambda|+R)\le |z|+R+\frac{|\lambda|+R}{2},\]
    provided $R$ is sufficiently large.

    Regarding \ref{rigiid2}, we first show that outside a set of negligible probability with respect to $R'\ge R$, for every $z\in B_{R'}(0)$ and $T({R'})\le r\le {R'}$, we have that
    \begin{align}
        \label{eqe}
        \# E\cap B_r(z)\le  C_{L_0} \kappa(r),
    \end{align}
    where
    \[E=\{w\in\zc^{(L)}: \ S^{(L)}(w)\le S({R'}) \},\]
    and $C_{L_0}>0$ is a sufficiently large constant depending (only) on ${L_0}$.
    
    Proposition \ref{coroloca} will cover the small values of $r$ until the almost independence overtakes and we can use Proposition \ref{lemindep}. First assume that $T({R'})\le r \le T({R'})^7$ (and that $R'\ge R$ is sufficiently large to ensure that $T(R')^7<R'$). Let us cover $B_{2R'}(0)$ by $O(R'^2)$ (pairwise disjoint) squares $Q$ of side-length $T({R'})$. By Proposition \ref{coroloca} and a union bound, outside a set of negligible probability with respect to ${R'}$, the estimate $\#E\cap Q\le C_{L_0} \kappa(T({R'}))$ holds simultaneously for every square $Q$ from the covering and for $C_{L_0}$ sufficiently large. So for every $z\in B_{R'}(0)$, 
    \[\#E\cap B_r(z)\le C_{L_0} (r/T({R'}))^2 \kappa(T({R'}))\le  C_{L_0} \kappa(r).\]
    (Here, as usual, we allow $C_{L_0}$ to change from one occurence to the next.)
    It remains to check the case $r\ge T^7({R'})$.  
    Let $Q_0$ be a square of side-length $T^7({R'})$ intersecting $B_{2{R'}}(0)$, and cover it by a collection $\qc$ of $N:=\lceil T^6({R'})\rceil^2$ squares of side-length $T({R'})$.  Applying Proposition \ref{lemindep} with $\lambda= T^4({R'})$
    we see that outside a set of negligible probability with respect to ${R'}$,
    \begin{align*}
        \#\{Q\in\mathcal{Q}:\ \exists z\in\zc^{(L)}\cap Q, \ s^{(L)}(z)<\logg^{-1} {R'}\}\le  T^4({R'}).
    \end{align*}
    Here we used that for the chosen parameters, all conditions of Proposition \ref{lemindep} are satisfied provided $R'\ge R_{L_0}$ is sufficiently large.
    Since by \ref{rigid2} each $Q\in\qc$ contains at most $C_{L_0} T^2({R'})$ points in $\zc^{(L)}$, we get
    \begin{align*}
        \#\{z\in \zc^{(L)}\cap Q_0: \ s^{(L)}(z)<\logg^{-1} {R'}\}\le   C_{L_0} T^6({R'}).
    \end{align*}
    Notice that if $s^{(L)}(z)\ge \logg^{-1} {R'}$ for $z\in \zc^{(L)}\cap Q_0$, then using \ref{rigid2} again,
    \[S^{(L)}(z)\ge s^{(L)}(z)^{\#\zc^{(L)}\cap B_1(z)}\ge (\logg R')^{- C_{L_0} T^2(R')}> S(R'),\]
    provided that $R'\ge R_{L_0}$ is sufficiently large.
    In particular,
    \[\#E\cap Q_0\le  C_{L_0} T^6({R'}).\]
    By a union bound, outside a set of negligible probability with respect to ${R'}$, for every square $Q$ of side-length $T^7({R'})$ intersecting $B_{2{R'}}(0)$ we have that 
    \[\#E\cap Q\le  C_{L_0} T^6({R'}).\] 
    So for every $z\in B_{R'}(0)$, 
    \[\#E\cap B_r(z)\le C_{L_0} (r/T^7({R'}))^2  T^6({R'})\le C_{L_0} r^2/T^8({R'}) \le C_{L_0} \kappa(r). \]
    This proves \eqref{eqe}.
    
    We assume from now on that \eqref{eqe} holds for every $R_k=e^k R$ with $k\ge 0$, since by a union bound this is satisfied outside a set of negligible probability with respect to $R$. Let $z\in\cc$ and $r\ge T(e|z|+R)$ and suppose first that $|z|>2r$. Define
    \[k_0=\min\{k\ge 0 : \ |z|< R_k\}.\]
    Notice that $z\in B_{R_{k_0}}(0)$ and $r\le R_{k_0}$. Moreover, 
    \[R_{k_0}=
    \begin{cases}
    R & \text{if } k_0=0
    \\ e R_{k_0-1}\le e |z| & \text{if } k_0\ge 1
    \end{cases}
    ,\]
    So, 
    \[T(R_{k_0})\le  T(e|z|+R)\le r.\]
    This allows us to use \eqref{eqe} for $z$, $r$ and $R'=R_{k_0}$. Since we are assuming that $|z|>2r$, we get
    \begin{align}
    \label{eq5}
    \#\Big\{w\in\zc^{(L)}\cap B_r(z): & \ S^{(L)}(w) \le S(2e|w|+R) \Big\}
       \\ & \le \#\Big\{w\in\zc^{(L)}\cap B_r(z):  \ S^{(L)}(w)\le S(e|z|+R) \Big\} \notag
       \\ &\le \#\Big\{w\in\zc^{(L)}\cap B_{r}(z):  \ S^{(L)}(w)\le S(R_{k_0}) \Big\}\le C_{L_0} \kappa(r). \notag
    \end{align}

It remains to check the case $|z|\le 2r$. Since $B_r(z)\subseteq B_{3r}(0)$ and $\kappa(3r)\le C\kappa(r)$, we can assume without loss of generality that $z=0$. If $r\le R$, then we can use \eqref{eqe} for $z=0$, $r$ and $R$ since $r\ge T(R)$. We get
\begin{align}
\label{eq6}
    \#\Big\{w\in\zc^{(L)}\cap B_r(0): & \ S^{(L)}(w) \le S(2e|w|+R) \Big\}
       \\ & \le \#\Big\{w\in\zc^{(L)}\cap B_r(0):  \ S^{(L)}(w)\le S(R) \Big\}
        \Big\}\le C_{L_0} \kappa(r). \notag
\end{align}
On the other hand, if $r\ge R$ define
    \begin{align*}
        k_1=\min\{k\ge 1 : \ r\le R_k\}.
    \end{align*}
Applying \eqref{eqe} for $z=0$, $r=R_j$ and $R'=R_{j}$ for every $0\le j\le k_1$,
    \begin{align*}
        \#\{w\in \zc^{(L)}\cap  & B_r(0):  \ S^{(L)}(w) \le S(2e|w|+R) \}
        \\ &\begin{multlined}[t]
            \le \#\{w\in\zc^{(L)}\cap B_{R}(0):  \ S^{(L)}(w) \le S(2e|w|+R) \} \\ + \sum_{j=1}^{k_1} \#\{w\in\zc^{(L)}\cap B_{R_{j}}(0)\smallsetminus B_{R_{j-1}}(0):  \ S^{(L)}(w) \le S(2e|w|+R) \}
        \end{multlined}
        \\ &\begin{multlined}[t]
            \le \#\{w\in\zc^{(L)}\cap B_{R}(0):  \ S^{(L)}(w) \le S(R) \} \\ + \sum_{j=1}^{k_1} \#\{w\in\zc^{(L)}\cap B_{R_{j}}(0)\smallsetminus B_{R_{j-1}}(0):  \ S^{(L)}(w) \le S(R_j) \}
        \end{multlined}
        \\ & \le C_{L_0}\sum_{j=0}^{k_1} \frac{e^{2j}R^2}{\logg^2(e^j R)}
        \le C_{L_0}\kappa(R_{k_1})\le C_{L_0} \kappa(r).
    \end{align*}
Combining this with \eqref{eq5} and \eqref{eq6}, we obtain \ref{rigiid2}.
\end{proof}

We can now prove our main theorem.
\begin{proof}[Proof of Theorem \ref{teofull}]
    The result follows from applying Proposition \ref{propseip}, using the functions $T(2er+R)$ and $S(2er+R)$ provided by Lemma \ref{lemprob}. A straightforward computation shows that all the requirements for Proposition \ref{propseip} are met.
\end{proof}

For the proof of Theorem \ref{teosam}, we need the following Bessel inequality.
\begin{proposition}\label{propbe}
    (Bessel inequality) Let $d\in \nn$, $1\le p<\infty$ and $1< L <2$.
    There is a constant $C=C(p)$ such that outside a set of negligible probability (in terms of $d$), for every polynomial $f$ of degree $\le d$,
    \[\sum_{z\in\zc^{(L)}}|f(z)|^pe^{-p|z|^2/2}\le C \sqrt{\logg d} \|f\|_{p}^p.\]   
\end{proposition}
\begin{proof}
Let $R$ be such that $R^2=d+3$. From \cite[Theorem 1]{Kr} and a union bound it follows that outside a set of negligible probability for every $z\in B_{5R}(0)$,
\[\#\zc^{(L)}\cap B_1(z)\le L \sqrt{\logg d}\le 2 \sqrt{\logg d}.\] 
Also, outside a set of negligible probability, the statements in Lemma \ref{lemprob} hold. In particular, for every $z\in B_{5R}(0)^c$ we have
\[\#\zc^{(L)}\cap B_1(z)\le C \logg |z|.\]  
For every $n\in\nn$, let $r_n=2^{n+1}R$ and $A_n=B_{r_{n+1}}(0)\smallsetminus B_{r_n}(0)$. From \eqref{eqloc} applied to each $z\in\zc^{(L)}$, we get
\begin{multline*}
    \sum_{z\in\zc^{(L)}}|f(z)|^pe^{-p|z|^2/2}
    \le C \sqrt{\logg d} \int_{B_{4R}(0)}|f(z)|^pe^{-p|z|^2/2} dA(z) \\ + C \sum_{n=1}^\infty  \logg(r_n)\int_{A_n}|f(z)|^pe^{-p|z|^2/2} dA(z).
\end{multline*}
Applying Lemma \ref{lemtai},
\begin{align}
\label{eqbe}
    \sum_{z\in\zc^{(L)}}|f(z)|^pe^{-p|z|^2/2}
    &\le C \sqrt{\logg d} \|f\|_{p}^p  + C\big(\frac{e}{d}\big)^{d/2} \|f\|_{p}^p \sum_{n=1}^\infty   \logg(r_n) r_n^{d+2}  e^{- r_n^2/2}.
\end{align}
Notice that
\begin{align*}
    \big(\frac{e}{d}\big)^{d/2}\sum_{n=1}^\infty   \logg(r_n) r_n^{d+2}  e^{- r_n^2/2} 
    &\le
    \big(\frac{ e}{d}\big)^{d/2}\int_{2R}^\infty x^{d+2}  e^{- x^2/2}dx
    \\ &\le C
    \big(\frac{ 2e}{d}\big)^{d/2}\int_{\sqrt{2}R}^\infty x^{d+2}  e^{- x^2}dx
    \\&\le C\big(\frac{ 2e}{d}\big)^{d/2}\int_{2 R^2}^\infty y^{(d+1)/2}  e^{- y}dy
    \\&\le C d^{2}(1-\lambda_{\big\lceil \tfrac{d+1}{2}\big\rceil}(\sqrt{2}R)),
\end{align*}
where in the last step we used \eqref{eqei} and Stirling's approximation formula. Since $R^2=d+3\ge 2\lceil \tfrac{d+1}{2}\rceil$, combining this with \eqref{eqbe} and \eqref{eqche} we arrive at
\begin{align*}
    \sum_{z\in\zc^{(L)}}|f(z)|^p&e^{-p|z|^2/2}
    \\ &\le C \|f\|_{p}^p \Big(\sqrt{\logg d}  + d^{2} e^{-\big(2R^2-\big\lceil \tfrac{d+1}{2}\big\rceil-\big\lceil \tfrac{d+1}{2}\big\rceil\log\big(2R^2\big\lceil \tfrac{d+1}{2}\big\rceil^{-1}\big)\big)}\Big)
    \\ &\le C \|f\|_{p}^p \Big(\sqrt{\logg d}  + d^{2} e^{-(3-\log 4)\big\lceil \tfrac{d+1}{2}\big\rceil}\Big). \qedhere
\end{align*}
\end{proof}

\begin{proof}[Proof of Theorem \ref{teosam}]
As usual, we allow all the involved constants to depend on $p$.
Proposition \ref{propbe} gives the inequality on the right-hand side. For the left-hand side, we apply Theorem \ref{teofull} for $R=R_d$ (recall that $R_d^2=d+C_1\sqrt{d\logg d}$). We can assume that $d\ge C_2$ since for $d\le C_2$ the failure probability can be chosen as 1 so that the claim trivially holds. Since $L_d-1\ge\tfrac{C_3}{\logg\logg d}$, we see that outside of an event of negligible probability in terms of $d$, for every polynomial $f$ of degree $\le d$,
\begin{align}\label{eqparti}
    \|f\|_{p}^p&\le \begin{multlined}[t]
        e^{\sqrt{\logg d}}\omega(2R_d) \sum_{z\in \zc^{(L_d)}\cap B_{R_d}(0)}|f(z)|^pe^{-p|z|^2/2}
        \\ + e^{\sqrt{\logg d}} \sum_{z\in \zc^{(L_d)}\cap B_{R_d}(0)^c}\omega(2|z|) |f(z)|^pe^{- p|z|^2/2}
    ,
    \end{multlined}
\end{align}
provided $C_3$ is sufficiently large.

From Lemma \ref{lemprob} we know that outside of a set of negligible probability,
\[\#\zc^{(L_d)}\cap B_1(z)\le C T(|z|+R_d)^2, \quad z\in\cc.\]
For every integer $n\ge -1$, let $r_n=2^{n}R_d$. 
Using 
\eqref{eqloc} and Lemma~\ref{lemtai},
\begin{align*}
    \sum_{z\in \zc^{(L_d)}\cap B_{R_d}(0)^c}\omega(2|z|) &|f(z)|^p
    e^{- p|z|^2/2}  
    \\ &\le C
    \sum_{n=1}^\infty T^2(r_{n+1})\omega(r_{n+1})\int\limits_{r_{n-1}-1\le |w|\le r_n +1} |f(w)|^pe^{- p|w|^2/2}  dA(w)
    \\& \le C
    \sum_{n=1}^\infty r_{n}^{3}\big(\frac{ r_n^2}{d}\big)^{d/2} e^{-( r_{n-2}^2-d)/2}\|f\|_{p}^p
    \\ &\le C \sum_{n=1}^\infty 2^{(3+d)n}e^{-d(2^{2n{-4}}-1)/2} R_d^{3}\big(\frac{ R_d^2}{d}\big)^{d/2} e^{-( R_d^2-d)/2}\|f\|_{p}^p.
\end{align*}
Notice that $(1+1/x)^x\le e^{1-1/4x}$ for every $x\ge 1$. We use this for 
\[x=\frac{d}{R_d^2-d}=\frac{\sqrt{d}}{C_1\sqrt{\logg d}},\] 
where we are assuming that $d$ is large enough so that $x\ge 1$. Assuming $C_1$ is sufficiently large, we get
\begin{align*}
    \sum_{z\in \zc^{(L_d)}\cap B_{R_d}(0)^c}\omega(2|z|) |f(z)|^pe^{- p|z|^2/2} & 
    \le C \sum_{n=1}^\infty 2^{3n}e^{-d(2^{2n{-4}}-2n-1)/2} R_d^{3} e^{-\tfrac{( R_d^2-d)^2}{8d}}\|f\|_{p}^p
    \\ &\le C \sum_{n=1}^\infty 2^{3n}e^{-(2^{2n{-4}}-2n-1)} d^{3} e^{-C_1\logg d{/8}}\|f\|_{p}^p
     \\ &\le C  d^{3-C_1/8} \|f\|_{p}^p.
\end{align*}
Combining this with \eqref{eqparti} and assuming that $C_1$ and $d$ are sufficiently large we obtain
\begin{align*}
    \|f\|_{p}^p&\le Ce^{\sqrt{\logg d}}\omega(2R_d) \sum_{z\in \zc^{(L_d)}\cap B_{R_d}(0)}|f(z)|^pe^{- p|z|^2/2}+\frac{\|f\|_{p}^p}{2}.
\end{align*}
Rearranging the inequality we get the desired result.
\end{proof}

We finish this section by showing the optimality up to $\logg\logg$ terms of the weight $\omega$ in Theorem \ref{teofull}.

\begin{rem}[Optimality of the weight] \label{remopt}
Assume, for example, that Theorem \ref{teofull} holds with $\logg \omega (r)=\sqrt{\logg r}/\logg\logg r$. Let $\tau>0$. By \cite[Theorem 1]{SoTs3}, the probability that $\zc^{(L)} \cap B_\tau(0)=\varnothing$ is bigger than $e^{-D\tau^4}$ for some constant $D>0$. 
One can then show that there exist $\tau_0,\widetilde{D}>0$ such that with high probability for every $\tau>\tau_0$, there exists $a\in\cc$ with $|a|\le e^{\widetilde{D}\tau^4}$ such that $\zc^{(L)}\cap B_\tau(a)=\varnothing$. This follows from the almost independence over long distances using an analogous argument to the proof of Proposition~\ref{lemindep}, so we skip the details. 
Now consider $f(z)=\mathcal{T}_a1(z)=e^{\overline{a}z-|a|^2/2}$. It is clear that
\[\|f\|_{p}=\|1\|_{p}=1.\]
On the other hand, from Lemma \ref{lemprob} we know that outside of a set of negligible probability,
\[\#\zc^{(L)}\cap B_1(z)\lesssim T(|z|+R)^2, \quad z\in\cc.\]
So by \eqref{eqloc},
\begin{align*}
    1&=\|f\|_{p}\lesssim \sum_{z\in\zc^{(L)}}\omega(|z|+R)|f(z)|^p e^{-p|z|^2/2}
    =\sum_{z\in\zc^{(L)}}\omega(|z|+R) e^{-p|z-a|^2/2}
    \\ &\lesssim \int_{B_{\tau-1}(a)^c}T(|z|+R)^2 \omega(|z|+R) e^{-p|z-a|^2/2} dA(z)
    \\ &\lesssim e^{\frac{C\sqrt{\logg (|a|+R)}}{\logg\logg (|a|+R)}}\int_{B_{\tau-1}(a)^c} e^{\frac{C\sqrt{\logg |z-a|}}{\logg\logg |z-a|}} e^{- p|z-a|^2/2} dA(z)
    \\ &\lesssim  e^{\frac{C\tau^2}{\logg \tau}} e^{-c\tau^2} \xrightarrow[]{\tau\to\infty} 0,
\end{align*}
leading to a contradiction. We showed that with high probability the weight cannot be significantly improved.

Regarding the sharpness in Theorem \ref{teosam}, it is easy to show something weaker, namely, that with non-negligible probability, the sampling constants cannot be significantly improved. We discuss the general idea and skip the details. For the left-hand side inequality in Theorem \ref{teosam}, one can proceed similarly as before and notice that with non-negligible probability $\zc_d \cap B_\tau(0)=\varnothing$, where $\tau=(\logg d)^{1/4}/\logg \logg d$. Taking $f=1$ and proceeding as before yields the sharpness up to $\log\log$ terms in the exponent of the sampling constant $A_d$. For the right-hand side inequality in Theorem \ref{teosam}, notice that by \cite[Theorem~1]{Kr}  $\#\zc_d \cap B_1(0)\ge \sqrt{\logg d}/\logg\logg d$ with non-negligible probability. Taking $f=1$ as before ensures the sharpness up to $\log\log$ terms of the sampling constant $B_d$. Naturally, we expect this to hold with high probability by relying on almost independence as before. However, an extra truncation argument would be needed since $f(z)=\mathcal{T}_a1(z)=e^{\overline{a}z-|a|^2/2}$ is not a polynomial. We did not pursue this further.
\end{rem}

\appendix
\section{Proof of Lemma \ref{lemlow}}\label{applem}

\begin{proof}[Proof of Lemma \ref{lemlow}]
    We assume that $1<L<2$ and track the dependence of the constants. Otherwise, the same argument yields constants depending on $L$.

    Notice that the separation condition \ref{its} ensures the estimate $\#E\cap B_r(0)\le \beta  r^2/\logg^2 r$ is satisfied for every $r\ge r_0$ since $r_0=2T(r_0)\ge T(0)$. The proof only makes use of this fact and the bound for $S$ in terms of $T$, rather than \ref{its}.
    
    Let us begin by showing \eqref{eqsup}.
    Notice that for $(m,n)\in A$, 
    \begin{align*}
        |\lambda_{mn}|\le |\lambda_{mn}-z_{mn}| + |z_{mn}|\le T( |\lambda_{mn}|) + |z_{mn}|\le |\lambda_{mn}|/2 + |z_{mn}|,
    \end{align*}
    In particular, 
    \begin{align}
    \label{eqlamt}
        |z_{mn}|\ge|\lambda_{mn}|/2 \ge T( |\lambda_{mn}|).
    \end{align}
    Now since $|z|\le \sqrt{2\pi/L}$, then for $m^2+n^2\ge 32$ sufficiently far away from 0,
    \[\Big|\frac{z}{z_{mn}}\Big|\le 2\Big|\frac{z}{\lambda_{mn}}\Big|
    \le \frac{2\sqrt{2}}{\sqrt{m^2+n^2}}\le\frac{1}{2}.\]
    Using again that $|z|\le \sqrt{2\pi/L}\le \sqrt{2\pi}$,
    \begin{align*}
        \Big|\log\Big(1-\frac{z}{z_{mn}}\Big)&+ \frac{z}{z_{mn}}+\frac{z^2}{2\lambda_{mn}^2}\Big|
        = \Big|\sum_{k=3}^\infty \frac{1}{k}\frac{z^k}{z_{mn}^k}+\frac{z^2}{2z_{mn}^2}-\frac{z^2}{2\lambda_{mn}^2}\Big|
        \\ &\le C \frac{1}{|z_{mn}|^3} + CT( |\lambda_{mn}|)\frac{|z_{mn}|+|\lambda_{mn}|}{|z_{mn}\lambda_{mn}|^2}\le \frac{CT( |\lambda_{mn}|)}{|\lambda_{mn}|^3}.
    \end{align*}
    The sum in $m$ and $n$ of the right-hand side term can be estimated in terms of
    \[\int_1^\infty \frac{T(x)}{x^2} dx\le \int_1^\infty \frac{T(1)+\logg x}{x^2} dx \le 2 T(1),\]
    which gives \eqref{eqsup}.
    
    For the remaining inequalities, we proceed as follows. Let $h(z)$ be given by
    \begin{align*}
        h(z)=\frac{g(z)}{\upsigma(z)}.
    \end{align*}
    By \eqref{eqsig},
    \begin{align*}
        |e^{-\tfrac{L}{2}|z|^2}g(z)|\ge c d(z,\Lambda) |h(z)|,
    \end{align*}
    so it suffices to establish a lower bound for $h(z)$. To achieve this, given $z\in \cc$, decompose $A=A_1\cup A_2$, where
    \[A_1=\{(m,n)\in A: \ |\lambda_{mn}|\le 4|z|\} \quad \text{ and} 
    \quad A_2=A \setminus A_1.\]
    Also write $A^c=B_0\cup B_1\cup B_2$ where
    \begin{align*}
        B_0&=\{(m,n)\in\zz^2: \ |\lambda_{mn}|< r_0\},
        \\ B_1&=\{(m,n)\in\zz^2: \ r_0\le |\lambda_{mn}|< 2|z|,\ z_{mn}\in E\},
        \\ B_2&=\{(m,n)\in\zz^2: \ \max\{r_0,2|z|\}\le |\lambda_{mn}|,\ z_{mn}\in E\}.
    \end{align*}
    Factorize $h$ as
    \begin{align}
        \label{eqh}
        h(z)=\frac{h_1(z)h_2(z)}{\upsigma_0(z)\upsigma_1(z)\upsigma_2(z)},
    \end{align}
    where for $i=1,2$,
    \begin{align*}
        \upsigma_0(z)&=z  \sideset{}{'}\prod_{B_0} \big(1-\frac{z}{\lambda_{mn}}\big)e^{\frac{z}{\lambda_{mn}}+\frac{1}{2}\frac{z^2}{\lambda_{mn}^2}},
        \\ \upsigma_i(z)&=  \prod_{B_i} \big(1-\frac{z}{\lambda_{mn}}\big)e^{\frac{z}{\lambda_{mn}}+\frac{1}{2}\frac{z^2}{\lambda_{mn}^2}}, \quad \text{and,}
        \\ h_i(z)&=\prod_{A_i}\frac{1-z/z_{mn}}{1-z/\lambda_{mn}}e^{ \frac{z}{z_{mn}}-\frac{z}{\lambda_{mn}}}.
    \end{align*}
    We need to estimate the absolute value of these functions, which we split in several steps.
    
    \textbf{Part 1. Bound for $|h_1|$.} When dealing with $h_1$ we can assume that $4|z|\geq r_0$ since it is defined in terms of $A_1$. From \eqref{eqlamt},
    \begin{align*}
        \big|\frac{1}{z_{mn}}-\frac{1}{\lambda_{mn}}\big|
        =\big|\frac{\lambda_{mn}-z_{mn}}{z_{mn}\lambda_{mn}}\big|
        \le \frac{2T(|\lambda_{mn}|)}{|\lambda_{mn}|^2}.
    \end{align*}
    A simple computation now leads to 
    \begin{align}
    \label{eqexp}
        \Big| \sum_{A_1} \big(\frac{1}{z_{mn}}-\frac{1}{\lambda_{mn}}\big)\Big|
        &\le \sum_{A_1} \frac{2T(|\lambda_{mn}|)}{|\lambda_{mn}|^2} \le C \int_{r_0}^{4|z|} \frac{T(x)}{x} dx 
       \le C \int_{T(1)}^{4|z|} \frac{T(1)+\logg x}{x} dx
        \\ & \le C T(1) \logg(|z|/T(1))+ C\logg^2(|z|).  \notag
    \end{align}
    This deals with the exponentials in the definition of $h_1$.
    Note that
    \[\frac{1-z/z_{mn}}{1-z/\lambda_{mn}}
    =\frac{\lambda_{mn}}{z_{mn}}\cdot \frac{z_{mn}- z}{\lambda_{mn} - z}
    =\frac{1-(\lambda_{mn}- z_{mn})/(\lambda_{mn}- z)}{1-(\lambda_{mn}- z_{mn})/\lambda_{mn}}.\]
    Thus,
    \begin{multline}
    \label{eqprod}
        \Big|\prod_{A_1}\frac{1-z/z_{mn}}{1-z/\lambda_{mn}}\Big|
        \\ \ge \prod_{\substack{A_1 \\ |\lambda_{mn}- z|<2T(|\lambda_{mn}|)}} \frac{|\lambda_{mn}||z_{mn}- z|}{|z_{mn}||\lambda_{mn}- z|}
        \prod_{\substack{A_1 \\ |\lambda_{mn}- z|\ge 2T(|\lambda_{mn}|)}}\frac{1-T(|\lambda_{mn}|)/|\lambda_{mn}- z|}{1+T(|\lambda_{mn}|)/|\lambda_{mn}|}.
    \end{multline}
    For the first product of the right-hand side we have
    \begin{align}
    \label{eqprod2}
        \prod_{\substack{A_1 \\ |\lambda_{mn}- z|<2T(|\lambda_{mn}|)}} &\frac{|\lambda_{mn}||z_{mn}- z|}{|z_{mn}||\lambda_{mn}- z|}
        \ge \prod_{\substack{A_1 \\ |\lambda_{mn}- z|<2T(|\lambda_{mn}|)}} \frac{|\lambda_{mn}||z_{mn}- z|}{(|\lambda_{mn}|+T(|\lambda_{mn}|))|\lambda_{mn}- z|}
        \\ &\ge  \prod_{\substack{A_1 \\ |\lambda_{mn}- z|<2T(|\lambda_{mn}|)}} 
        \frac{2}{3}\frac{|z_{mn}- z|}{|\lambda_{mn}- z|}\notag
        \\ &\ge \frac{1}{d(z,\Lambda)} \exp(-CT^2(4|z|))\prod_{\substack{A_1 \\ |\lambda_{mn}- z|<2T(|\lambda_{mn}|)}} |z_{mn}- z| \notag
        \\ &\ge \frac{d(z,\zc)}{d(z,\Lambda)} \exp(-CT^2(4|z|))\prod_{{\substack{A_1 \\ z_{mn}\in B_1(z_{m'n'})\smallsetminus\{z_{m'n'}\}}}} \frac{|z_{mn}- z_{m'n'}|}{2} \notag
        \\ &\ge \frac{d(z,\zc)}{d(z,\Lambda)} S(|z_{m'n'}|) \exp(-CT^2(4|z|)) \notag
        \\ &\ge \frac{d(z,\zc)}{d(z,\Lambda)} S(|z|+1) \exp(-CT^2(4|z|)), \notag
    \end{align}
    where $z_{m'n'}$ is the closest point to $z$ with $(m',n')\in A_1$ (if there is more than one such point we may arbitrarily choose one). Here we are assuming $\{z_{mn}\}_{(m,n)\in A_1}\cap B_1(z)\neq \varnothing$, since otherwise we can simply omit the term $ S(|z_{m'n'}|)$.
    For the second product of the right-hand side of \eqref{eqprod} we have
    \begin{multline*}
        \prod_{\substack{A_1 \\ |\lambda_{mn}- z|\ge 2T(|\lambda_{mn}|)}}\frac{1-T(|\lambda_{mn}|)/|\lambda_{mn}- z|}{1+T(|\lambda_{mn}|)/|\lambda_{mn}|} \\ \ge 
        \exp \Big(-C\sum_{r_0\le |\lambda_{mn}- z|\le 5|z|}\frac{T(|\lambda_{mn}|)}{|\lambda_{mn}-z|} -C\sum_{A_1} \frac{T(|\lambda_{mn}|)}{|\lambda_{mn}|}\Big)
         \ge \exp(-CT(4|z|)|z|).
    \end{multline*}
    Combining this with \eqref{eqexp} and \eqref{eqprod2}, and recalling that $T(x)\le T(1)+\logg x$ we get
    \begin{align}
    \label{eqh1}
        |h_1(z)|\ge \frac{d(z,\zc)}{d(z,\Lambda)} S(|z|+1) \exp(-C(T^2(1)+T(1)|z|\logg(|z|/T(1))+|z|\logg^2|z|)).
    \end{align}
    
    \textbf{Part 2. Bound for $|h_2|$.} Recall that for $(m,n)\in A_2$, $|\lambda_{mn}| \geq 4|z|$. Combining this with \eqref{eqlamt} yields
    \begin{align*}
        \left|\frac{z}{z_{mn}}\right|\le 2 \left|\frac{z}{\lambda_{mn}}\right| \le \frac{1}{2}.
    \end{align*}
    Now observe that
    \begin{align*}
        \Big|\log\Big(1-\frac{z}{z_{mn}}\Big)&-\log\Big(1-\frac{z}{\lambda_{mn}}\Big) + \frac{z}{z_{mn}}-\frac{z}{\lambda_{mn}}\Big|
        = \Big|\sum_{k=2}^\infty \frac{1}{k} \Big(\Big(\frac{z}{z_{mn}}\Big)^k-\Big(\frac{z}{\lambda_{mn}}\Big)^k\Big)\Big|
        \\ &=\Big|\Big(\frac{z}{z_{mn}}-\frac{z}{\lambda_{mn}}\Big)\sum_{k=1}^\infty \frac{z^k}{k+1} \sum_{j=0}^k \frac{1}{z_{mn}^j\lambda_{mn}^{k-j}}\Big|
        \\ &\le \frac{2|z|T(|\lambda_{mn}|)}{|z_{mn}\lambda_{mn}|}\sum_{k=1}^\infty  \Big|\frac{2z}{\lambda_{mn}}\Big|^k\le \frac{CT(|\lambda_{mn}|)|z|^2}{|\lambda_{mn}|^3}.
    \end{align*}
    Thus,
    \begin{align}
    \label{eqh2}
        |h_2(z)|&\ge \exp\Big(-C|z|^2\sum_{A_2} \frac{T(|\lambda_{mn}|)}{|\lambda_{mn}|^3}\Big) \ge \exp\Big(-C|z|^2\int_{\max\{4|z|,r_0\}}^\infty \frac{T(x)}{x^2} dx\Big)
        \\& \ge \exp\Big(-C|z|^2\int_{4|z|}^\infty \frac{T(1)+\logg x}{x^2} dx\Big)\ge \exp(-C|z|(T(1)+\logg|z|)). \notag
    \end{align}
    
    \textbf{Part 3. Bound for $|\upsigma_0|$.} First notice that by a symmetry argument,
    \[\sideset{}{'}\sum_{B_0}\frac{1}{\lambda_{mn}}=\sideset{}{'}\sum_{B_0}\frac{1}{\lambda_{mn}^2}=0,\]
    as a result,
    \begin{align*}
        \sideset{}{'}\prod_{B_0} e^{\frac{z}{\lambda_{mn}}+\frac{1}{2}\frac{z^2}{\lambda_{mn}^2}}=1.
    \end{align*}
    Plugging this into the definition of $\upsigma_0$ we arrive at
    \begin{align}
       \label{eqsig1}
        \begin{split}
            |\upsigma_0(z)|=\Big|z  \sideset{}{'}\prod_{B_0} \big(1-\frac{z}{\lambda_{mn}}\big)\Big| 
            &\le |z| \sideset{}{'}\prod_{B_0} (1+\sqrt{L/\pi}|z|)\\
            &\le |z| \sideset{}{'}\prod_{B_0} (1+\sqrt{2/\pi}|z|)\\
            &\le \exp(Cr_0^2 \logg|z|)\le \exp(CT(1)^2 \logg|z|),
        \end{split}
    \end{align}
where in the last step we used that 
\begin{align}
\label{eqr0}
    r_0\le 8 T(1),
\end{align}
since 
\[r_0=2 T(r_0)\le 2(T(1)+\log r_0)\le 2(T(1)+r_0/e).\]

    \textbf{Part 4. Bound for $|\upsigma_1|$.}  When dealing with $\upsigma_1$ we can assume that $2|z|\geq r_0$ since it is defined in terms of $B_1$. Let $\{B_{1,k}\}_{k=0}^K$ be the partition of $B_1$ given by
    \[B_{1,k}=\{(m,n)\in B_1 : 2^{k}r_0\le |\lambda_{mn}|<2^{k+1}r_0 \},\]
    where $2^{K-1}r_0<|z|\le 2^{K}r_0$.  Note that for $(m,n)\in B_{1,k}$, $|\lambda_{mn}|\ge r_0$ and so,
    \[|z_{mn}|\le |\lambda_{mn}|+T(|\lambda_{mn}|)\le \tfrac{3}{2}|\lambda_{mn}|< 2^{k+2}r_0.\]
Consequently, for every $0\le k\le K$,
    \[\# B_{1,k}\le \# E\cap B_{2^{k+2}r_0}(0)\le \beta  4^{k+2} r_0^2(\logg (2^{k+2}r_0))^{-2}.\]
    With this at hand we see that
    \begin{align}
    \label{eqrec}
        \sum_{B_1} \frac{1}{|\lambda_{mn}|} &\le 16 \beta  \sum_{k=0}^K   \frac{4^{k}r_0^2(\logg (2^{k+2}r_0))^{-2}}{2^k r_0} \le  C \beta   2^{K+1} r_0(\logg r_0)^{-2} \le C \beta  |z| (\logg r_0)^{-2}.
    \end{align}
    Similarly,
    \begin{align*}
        \sum_{B_1} \frac{1}{|\lambda_{mn}|^2} &\le 4\beta  \sum_{k=0}^K  (\logg (2^{k+2}r_0))^{-2} \le 4\beta  \sum_{k=0}^\infty  \frac{1}{((k+2) \log 2+\log r_0)^2}
        \le \frac{C \beta } {\logg r_0}.
    \end{align*}
    Combining this with \eqref{eqr0} and \eqref{eqrec}, 
    \begin{align}
        \label{eqsig2}
        |\upsigma_1(z)| &\le   \prod_{B_1} \Big|\frac{\lambda_{mn}-z}{\lambda_{mn}}\Big| \exp(C \beta  |z|^2/\logg r_0)
        \le (\frac{3|z|}{r_0})^{\# B_1} \exp(C \beta   |z|^2/\logg T(1))
        \\ &\le (\frac{3}{2}|z|)^{C \beta   |z|^2/\logg^2(3|z|)} \exp(C \beta   |z|^2/\logg T(1))
        \le \exp(C \beta  |z|^2/\logg T(1)). \notag
    \end{align}
    
    \textbf{Part 5. Bound for $|\upsigma_2|$.} Similar to before, let $\{B_{2,k} \}_{k\ge k_0}$ be the partition of $B_2$ given by
    \[B_{2,k}=\{(m,n)\in B_2 : 2^{k}|z|\le |\lambda_{mn}|<2^{k+1}|z| \},\]
    where $k_0\in\nn$ is the minimal index satisfying $2^{k_0+1}|z|\ge \max\{r_0,2|z|\}$. For $(m,n)\in B_{2,k}$ we have that $|\lambda_{mn}|\ge r_0$ and so,
    \[|z_{mn}|\le |\lambda_{mn}|+T(|\lambda_{mn}|)\le \tfrac{3}{2}|\lambda_{mn}|<  2^{k+2}|z| .\]
     Since $2^{k+2}|z|>2^{k_0+1}|z|\ge r_0$, for every $ k\in \nn$,
    \[\# B_{2,k}\le \# E\cap B_{ 2^{k+2}|z|}(0)\le  \beta   4^{k+2} |z|^2(\logg ( 2^{k+2}|z|))^{-2}.\]
    As a result,
    \begin{align*}
        \sum_{B_2} \frac{1}{|\lambda_{mn}|^3} &\le \frac{C \beta }{|z|} \sum_{k=k_0}^\infty  2^{-k} (\logg (2^{k+2}|z|))^{-2} \le \frac{C \beta  }{|z|\logg^2 r_0} \sum_{k=1}^\infty  2^{-k} \le  \frac{C \beta  }{|z|\logg^2 T(1)}.
    \end{align*}
    Consequently,
    \begin{multline*}
        \Big|\sum_{B_2}\log\Big(1-\frac{z}{\lambda_{mn}}\Big) + \frac{z}{\lambda_{mn}}+\frac{z^2}{2\lambda_{mn}^2}\Big|
        = \Big|\sum_{B_2} \sum_{\ell=3}^\infty \frac{1}{\ell}\Big(\frac{z}{\lambda_{mn}}\Big)^\ell\Big|
        \\ \le \sum_{B_2} \Big|\frac{z}{\lambda_{mn}}\Big|^3 \sum_{\ell=3}^\infty \frac{1}{\ell}\Big(\frac{1}{2}\Big)^{\ell-3}
        \leq \frac{C \beta   |z|^2}{\logg^2 T(1)}. 
    \end{multline*}

   Therefore,
    \begin{align}
        \label{eqsig3}
        |\upsigma_2(z)| \le \exp(C \beta  |z|^2(\logg T(1))^{-2}). 
    \end{align}

    \textbf{Part 6. Proof of \eqref{eqinf} and \eqref{eqinf2}.}
    From \eqref{eqsig} and \eqref{eqh}, we get
    \begin{align*}
        |e^{-\tfrac{L}{2}|z|^2}g(z)|\ge c d(z,\Lambda) |h(z)| =c d(z,\Lambda) \frac{h_1(z)h_2(z)}{\upsigma_0(z)\upsigma_1(z)\upsigma_2(z)}.
    \end{align*}
    Now applying the previous estimates for $\upsigma_i$ and $h_i$ from \eqref{eqh1}, \eqref{eqh2}, \eqref{eqsig1}, \eqref{eqsig2} and \eqref{eqsig3}, we obtain 
    \begin{align*}
        |e^{-\tfrac{L}{2}|z|^2}g(z)| &\ge c d(z,\zc) S(1 + |z|)
        \begin{multlined}[t]
             \exp\big(-C(T^2(1) \logg|z|+T(1)|z|\logg(|z|/T(1))
        \\ + |z|\logg^2|z| 
         + \beta  |z|^2/\logg T(1))\big)
         \end{multlined}
        \\& \ge c d(z,\zc)  \exp\big(-C \beta  (T^2(1) \logg^4 T(1)+|z|^2/\logg T(1) )\big),
    \end{align*}
    where the last step follows from estimating $S$ in terms of $T$ and bounding each term in the exponent according to whether $|z|$ is large or not. This proves \eqref{eqinf}. Noting that
    \[|e^{-\tfrac{L}{2}|z_{mn}|^2}g'(z_{mn})|=\lim_{z\to z_{mn}} \frac{|e^{-\tfrac{L}{2}|z|^2}g(z)|}{|z-z_{mn}|},\]
    and applying the previous bound then proves \eqref{eqinf2}.
\end{proof}
\bibliographystyle{abbrv}
\bibliography{biblio}

\end{document}